\documentclass[12pt,a4paper]{article}
\usepackage[utf8]{inputenc}
\usepackage{amsmath, amsthm}
\usepackage{amsfonts}
\usepackage{amssymb}
\usepackage{pdflscape}
\usepackage{rotating}
\usepackage{makeidx}
\usepackage{graphicx, hyperref}

\newtheorem{corollary}{Corollary}[section]
\newcommand{\Mod}[1]{\ (\mathrm{\it{mod}}\ #1)}
\newtheorem{theorem}{Theorem}[section]
\usepackage[width=150.00mm, height=186.00mm]{geometry}
\author{Ratan Lal and Vipul Kakkar\\ \textit{email: vermarattan789@gmail.com, vplkakkar@gmail.com}}
\title{Automorphism groups of groups of order $p^{2}q^{2}$}

\begin{document}
		\maketitle
	\begin{abstract}
		In this paper, we have computed the automorphism groups of all groups of order $p^{2}q^{2}$, where $p$ and $q$ are distinct primes.
	\end{abstract}
	\section{Introduction}
	A. S. Hadi, M. Ghorbani and F. N. Larki \cite{gpp2q2} classified the groups of order $p^{2}q^{2}$ upto isomorphism, where $p$ and $q$ are distinct primes. They showed that any group of order $p^{2}q^{2}$ is either a direct product of groups or a semidirect product of groups. Our aim in this paper is to find the automorphism group of all groups of order $p^{2}q^{2}$, where $p$ and $q$ are distinct primes. 
	
	Bidwell et. al. \cite{bcm2006} proved that if $H$ and $K$ are finite groups with no common direct factors and $G = H\times K$, then the automorphism group $Aut(G)$ of $G$ can be expressed as the $2\times 2$ matrices of maps satisfying some conditions (see \cite[Theorem 3.6, p. 487]{bcm2006}). Also, Bidwell and Curran \cite{bdw2006} proved that if $H$ and $K$ are finite group groups with a group homomorphism $\phi : K\longrightarrow Aut(H)$ and $G = H\rtimes_{\phi} K$ is the semidirect product of $H$ and $K$, then the automorphism group $Aut(G)$ of $G$ is in one to one correspondence with the group of $2\times 2$ matrices of maps satisfying some conditions (see \cite[Lemma 2.1, p. 489]{bdw2006}).  Recently, the automorphism groups of all groups of order $p^{2}q$ are obtained by Campedel et. al. \cite{p2q}. So, it motivates us to study the automorphism groups of all groups of order $p^{2}q^{2}$.
	
	 We will use the results mentioned in the previous paragraph in order to achieve our goal. In the second section, we have recalled some results which will be useful for us in achieving our goal of this paper. In the third section, we have listed all groups of order $p^{2}q^{2}$ upto isomorphism, where $p$ and $q$ are distinct primes. In the last section, we have found the structure of the automorphism group of all groups of order $p^{2}q^{2}$. 
	
	Throughout the paper, we will identify the internal direct product $G = HK$ (semidirect product $G = HK$) with the external direct product $G = H\times K$ (semidirect product $G = H\rtimes_{\phi} K$), where the group homomorphism $\phi: K \longrightarrow Aut(H)$ is defined by $\phi(k)(h) = h^{k} = k^{-1}hk$, for all $h\in H$ and $k\in K$. The commutator of two elements $x,y \in G$ is defined as $[x,y] = x^{-1}y^{-1}xy$ and the conjugate of an element $x$ is defined as $x^{y} = y^{-1}xy$. $\mathbb{Z}_{n},S_{n}, A_{n}$ and $D_{n}$ will denote the cyclic group of order $n$, symmetric group on $n$ symbols, alternating group on $n$ symbols and dihedral group of order $2n$ respectively. $Q_{8}$ is the quaternion group of order 8. $[\cdot]$ denotes the greatest integer function.
	
	\section{Preliminaries}
	In this section, we will recall some result about the automorphism group of direct product of groups and semidirect product of groups.
	
	\begin{theorem}\cite[Theorem 3.6, p. 487]{bcm2006}\label{dp}
		Let $G = H\times K$, where $H$ and $K$ have no common direct factor, and let 
		\[\begin{matrix}
		A =	\left\{\begin{pmatrix}
		\alpha & 0\\
		0 & 1
		\end{pmatrix} \mid \alpha \in Aut(H)\right\}, & B = \left\{\begin{pmatrix}
		1 & \beta\\
		0 & 1
		\end{pmatrix} \mid \beta \in Hom(K, Z(H))\right\},\\
		C = \left\{\begin{pmatrix}
		1 & 0\\
		\gamma & 1
		\end{pmatrix} \mid \gamma \in Hom(H, Z(K))\right\}, & D = \left\{\begin{pmatrix}
		1 & 0\\
		0 & \delta
		\end{pmatrix} \mid \delta \in Aut(K)\right\}.
		\end{matrix}\]
		
		Then $A, B, C, D$ are subgroups of $Aut(G)$ and $Aut(G) = ABCD$, where $AD = A\times D$ normalizes $B$ and $C$.
	\end{theorem}
	
	\begin{theorem}\cite[Lemma 2.1, p.489]{bdw2006}\label{t3}
		Let $G = H \rtimes K$ be the semidirect product of groups $H$ and $K$ and let
		\[\mathcal{A} = \left\{\begin{pmatrix}
		\alpha & \beta \\
		\gamma & \delta
		\end{pmatrix} \mid \begin{matrix}
		\alpha\in Map(H, H), & \beta \in Map(K, H)\\
		\gamma \in Hom(H,K), & \delta\in Hom(K,K)
		\end{matrix}.\right\}\]
		where $\alpha, \beta, \gamma, \delta$ satisfies the following properties for all $h\in H$ and $k\in K$,
		\begin{itemize}
			\item[($i$)] $\alpha(hh^{\prime}) =
			\alpha(h)\alpha(h^{\prime})^{\gamma(h)}$,
			\item[($ii$)]  $\beta(kk^{\prime}) = \beta(k)\beta(k^{\prime})^{\delta(k)}$,
			\item[$(iii)$] $\gamma(h^{k}) = \gamma(h)^{\delta(k)}$,
			\item[$(iv)$] $\alpha(h^{k})\beta(k)^{\gamma(h^{k})} = \beta(k)\alpha(h)^{\delta(k)}$,
			\item[$(v)$] For any $h^{\prime}k^{\prime}\in G$, there is a unique $hk\in G$ such that $\alpha(h)\beta(k)^{\gamma(h)} = h^{\prime}$ and $\gamma(h)\delta(k) = k^{\prime}$.
		\end{itemize}
		Then there is a one to one correspondence between the automorphism group of $G$, $Aut(G)$ and $\mathcal{A}$ given by $\theta \leftrightarrow \begin{pmatrix}
		\alpha & \beta\\
		\gamma & \delta
		\end{pmatrix}$, where $\theta(h) = \alpha(h)\gamma(h)$ and $\theta(k) = \beta(k)\delta(k)$. Also, if $\theta^{\prime} = \begin{pmatrix}
		\alpha^{\prime} & \beta^{\prime}\\
		\gamma^{\prime} & \delta^{\prime}
		\end{pmatrix}$, then 
		\[\theta^{\prime}\theta = \begin{pmatrix}
		\alpha^{\prime}\alpha + {\beta^{\prime}\gamma}^{\gamma^{\prime}\alpha} & \alpha^{\prime}\beta + {\beta^{\prime}\delta}^{\gamma^{\prime}\beta}\\ \gamma^{\prime}\alpha + \delta^{\prime}\gamma & \gamma^{\prime}\beta + \delta^{\prime} \delta
		\end{pmatrix}. \]
	\end{theorem}
	
	\begin{corollary}\cite[corollary 2.2, p. 490]{bdw2006}
		Let $G = H\rtimes K$ be the semidirect product of $H$ and $K$, where $K$ is abelian. Then
		\begin{itemize}
			\item[$(i)$] $\gamma \in Hom(H/[H,K], K)$,
			\item[$(ii)$] for all $h\in H$ and $k\in K$, $\alpha(h^{k})\beta(k)^{\gamma(h)} = \beta(k)\alpha(h)^{\delta(k)}$. 
		\end{itemize}
	\end{corollary}
	Let \begin{align*}
	P =& \{(\alpha,\delta) \in Aut(H)\times Aut(K)\mid \alpha(h^{k}) = \alpha(h)^{\delta(k)}, \forall h\in H, k\in K\},\\
	S =& \{\beta \in Map(K, H)\mid \beta(kk^{\prime}) = \beta(k)\beta(k^{\prime})^{k}, \forall k,k^{\prime}\in K\}
	\end{align*}
	and let the corresponding subsets of $\mathcal{A}$ be \\
	
	$R = \left\{\begin{pmatrix}
	\alpha & 0\\
	0 & \delta
	\end{pmatrix} \mid (\alpha,\delta) \in P\right\}$ and $Q = \left\{\begin{pmatrix}
	1 & \beta\\
	0 & 1
	\end{pmatrix} \mid \beta \in S\right\}.$
	\begin{theorem}\label{main}
		Let $G = H\rtimes K$ be the semidirect product of $H$ and $K$, where $K$ is abelian and $\begin{pmatrix}
		\alpha & \beta \\ \gamma & \delta
		\end{pmatrix} \in \mathcal{A}$. Then, if $\gamma = 0$, the trivial homomorphism, then $Aut(G)\simeq \mathcal{A} \simeq Q\rtimes R$.  
	\end{theorem}
	\begin{proof}
		Let $\theta = \begin{pmatrix}
		\alpha & \beta \\ \gamma & \delta
		\end{pmatrix} \in \mathcal{A}$. Then, if $\gamma = 0$, then we have 
		\[\begin{pmatrix}
		\alpha & \beta \\ 0 & \delta
		\end{pmatrix} = \begin{pmatrix}
		\alpha & 0 \\ 0 & 1
		\end{pmatrix}\begin{pmatrix}
		1 & \hat{\beta} \\ 0 & 1
		\end{pmatrix}\begin{pmatrix}
		1 & 0 \\ 0 & \delta
		\end{pmatrix}\in ABD,\]
		where $\hat{\beta} = \alpha^{-1}\beta\delta^{-1}$. Also, for any $\begin{pmatrix}
		1 & \beta \\ 0 & 1
		\end{pmatrix} \in Q$ and $\begin{pmatrix}
		\alpha & 0 \\ 0 & \delta
		\end{pmatrix}\in P\times S$, we have 
		\[\begin{pmatrix}
		\alpha & 0 \\ 0 & \delta
		\end{pmatrix}\begin{pmatrix}
		1 & \beta \\ 0 & 1
		\end{pmatrix}\begin{pmatrix}
		\alpha & 0 \\ 0 & \delta
		\end{pmatrix}^{-1} = \begin{pmatrix}
		1 & \alpha\beta\delta^{-1} \\ 0 & \delta
		\end{pmatrix}\in Q.\]
		Thus, $Q \triangleleft P\times S$. Also, $Q\cap (P\times S) = \{1\}$. Therefore, $\mathcal{A} \simeq Q\rtimes (P\times S)$.
	\end{proof}
	
	We will identify the automorphisms of the group $G$ with the corresponding matrices in $\mathcal{A}$.
	\section{Groups of order $p^{2}q^{2}$, $p$ and $q$ are distinct primes}
 In \cite{gpp2q2}, the authors have classified the groups of order $p^{2}q^{2}$, where $p$ and $q$ are distinct primes, as follows.
 \begin{theorem}[Theorem 3.1, p. 91]\cite{gpp2q2}\label{t1}
 	Let $G$ be a group of order $p^{2}q^{2}$. Then $G$ is isomorphic to one of the following groups.
 	\begin{itemize}
 		\item[($i$)] If $pq=6$, then $G$ is isomorphic to
 		\begin{enumerate}
 			\item  $\mathbb{Z}_{36}$,
 			\item $\mathbb{Z}_{18}\times \mathbb{Z}_{2}$,
 			\item $\mathbb{Z}_{6}\times \mathbb{Z}_{6}$,
 			\item  $\mathbb{Z}_{12}\times \mathbb{Z}_{3}$,
 			\item $\mathbb{Z}_{6}\times D_{6} \simeq S_{3}\times \mathbb{Z}_{6} \simeq D_{12}\times \mathbb{Z}_{3}$,
 			\item  $D_{6}\times D_{6}\simeq S_{3}\times S_{3}$,
 		\item  $D_{18}\times \mathbb{Z}_{2}$,
 		\item  $A_{4}\times \mathbb{Z}_{3}$,
 		\item  $D_{36}$,
 		\item $H\times \mathbb{Z}_{3}$,
 		\item $K\times \mathbb{Z}_{2}$,
 		\item  $\langle a,b,c \mid a^{3} = b^{3} = c^{4} = [a,b] = 1, c^{-1}ac = b, c^{-1}bc = a^{-1}\rangle$,
 		\item  $\langle a,b,c \mid a^{2} = b^{2} = c^{9} = [a,b] = 1, c^{-1}ac = b, c^{-1}bc = ab\rangle$,
 		\item  $\langle a,b,c \mid a^{3} = b^{3} = c^{4} = [a,b] = 1, c^{-1}ac = a^{-1}, c^{-1}bc = b^{-1}\rangle$,
  		\end{enumerate}
where
$$H = \langle a, b \mid a^{4} = b^{3} = 1, a^{-1}ba = b^{-1}\rangle$$
and
$$K = \langle a, b, c, \mid a^{2} = b^{2} = c^{2} = (abc)^{2} = (ab)^{3} = (ac)^{3} = 1\rangle.$$
\item[($ii$)] If $pq\ne 6$, then $G$ is isomorphic to
\begin{enumerate}
	\item[15.] $\mathbb{Z}_{q^{2}}\times \mathbb{Z}_{p^{2}}$,
	\item[16.] $\mathbb{Z}_{q}\times \mathbb{Z}_{q} \times \mathbb{Z}_{p^{2}}$,
	\item[17.]  $\mathbb{Z}_{q^{2}}\times \mathbb{Z}_{p} \times \mathbb{Z}_{p}$,
	\item[18.]  $\mathbb{Z}_{q}\times \mathbb{Z}_{q} \times \mathbb{Z}_{p}\times \mathbb{Z}_{p}$, 
	\item[19.] $\langle a,b \mid a^{q^{2}} = b^{p^{2}} = 1, a^{-1}ba = b^{r}, r^{q} \equiv 1 \Mod{p^{2}}$, where $q$ divides $p-1$,
	\item[20.]  $\langle a,b \mid a^{q^{2}} = b^{p^{2}} = 1, a^{-1}ba = b^{r}, r^{q^{2}} \equiv 1 \Mod{p^{2}}$, where $q^{2}$ divides $p-1$,
	\item[21.]  $\langle a,b,c \mid a^{q} = b^{q} = c^{p^{2}} = 1, [a,b] = [b,c] = 1, a^{-1}ca = c^{r}, r^{q} \equiv 1 \Mod{p^{2}}$, where $q$ divides $p-1$,
	\item[22.]  $\langle a,b,c \mid   a^{q^{2}} = b^{p} = c^{p} = [b,c] = 1, a^{-1}ba = b^{r}, a^{-1}ca = c^{r}, r^{q} \equiv 1 \Mod{p}$, where $q$ divides $p-1$,
	\item[23.]  $\langle a,b,c \mid a^{4} = b^{p} = c^{p} = [a,b] = [b,c] =1, a^{-1}ca = c^{-1}\rangle$,
	\item[24.] $\langle a,b,c \mid a^{q^{2}} = b^{p} = c^{p} = [a,b] = [b,c] = 1, a^{-1}ca = c^{r}, r^{q} \equiv 1 \Mod{p}$, where $q\ne 2$ divides $p-1$,
	\item[25.] $\langle a,b,c \mid a^{q^{2}} = b^{p} = c^{p} = [b,c] = 1, a^{-1}ba = b^{r}, a^{-1}ca = c^{r^{-1}}, r^{q} \equiv 1 \Mod{p}$, where $q\ne 2$ divides $p-1$,
	\item[26.] $\langle a,b,c \mid a^{q^{2}} = b^{p} = c^{p} = [b,c] = 1, a^{-1}ba = b^{r}, a^{-1}ca = c^{r^{-1}}, r^{q^{2}} \equiv 1 \Mod{p}$, where $q^{2}$ divides $p-1$,
	\item[27.] $\langle a,b,c \mid a^{q^{2}} = b^{p} = c^{p} = [b,c] = 1, a^{-1}ba = b^{r}, a^{-1}ca = c^{r^{n}}, r^{q} \equiv 1 \Mod{p}$, where $q\ne 2$ divides $p-1$ and $n\in \{2,3, \cdots, \frac{q-1}{2}\}$,
		\item[28.] $\langle a,b,c \mid a^{q^{2}} = b^{p} = c^{p} = [b,c] = 1, a^{-1}ba = b^{r}, a^{-1}ca = c^{r^{n}}, r^{q^{2}} \equiv 1 \Mod{p}$, where $q^{2}$ divides $p-1$ and $2 \le n \le \frac{q^{2}-1}{2}$ or $ n = mq (m \ge \frac{q+1}{2})$,
		\item[29.] $\langle a,b,c \mid a^{q^{2}} = b^{p} = c^{p} = [b,c] = 1, a^{-1}ba = b^{r}, a^{-1}ca = c^{r}, r^{q^{2}} \equiv 1 \Mod{p}$, where $q^{2}$ divides $p-1$,		
	\item[30.] $\langle a,b,c \mid a^{q^{2}} = b^{p} = c^{p} = [b,c] = 1, a^{-1}ba = b^{m}c^{n D}, a^{-1}ca = b^{n}c^{m}$, where $m + n\sqrt{D} = \sigma^{\frac{p^{2}-1}{q}}$, $\sigma$ is a primitive root of Galois field $GF(p^{2})$, $m, n, D\in GF(p)$, $n \ne 0$, $D$ is not a perfect square and $q$ divides $p+1$,
	\item[31.]  $\langle a,b,c \mid a^{4} = b^{p} = c^{p} = [b,c] = 1, a^{-1}ba = b^{m}c^{nD}, a^{-1}ca = b^{n}c^{m}$, where $m + n\sqrt{D} = \sigma^{\frac{p^{2}-1}{4}}$, $\sigma$ is a primitive root of Galois field $GF(p^{2})$, $\alpha, \beta, D\in GF(p)$, $\beta \ne 0$, $D$ is not a perfect square and $p\equiv 3 \Mod{4}$,
	\item[32.] $\langle a,b,c \mid a^{q^{2}} = b^{p} = c^{p} = [b,c] = 1, a^{-1}ba = b^{m}c^{nD}, a^{-1}ca = b^{n}c^{m}$, where $m + n\sqrt{D} = \sigma^{\frac{p^{2}-1}{q^{2}}}$, $\sigma$ is a primitive root of Galois field $GF(p^{2})$, $m, n, D\in GF(p)$, $\beta \ne 0$, $D$ is not a perfect square and $q^{2}$ divides $p+1$,
	\item[33.] $\langle a,b,c,d \mid a^{q} = b^{q} = c^{p} = d^{p} = [a,c] = [a,d] = [a,b] = [c,d] = 1, b^{-1}cb = c^{r}, b^{-1}db = d^{r}, r^{q} \equiv 1 \Mod{p}$, where $q$ divides $p-1$,
	\item[34.] $\langle a,b,c,d \mid a^{2} = b^{2} = c^{p} = d^{p} = [a,b] = [a,c] = [a,d] = [b,d] = [c,d] = 1, b^{-1}cb = c^{-1}\rangle$,
	\item[35.] $\langle a,b,c,d \mid a^{q} = b^{q} = c^{p} = d^{p} = [a,b] = [a,d] = [b,c] = [c,d] = 1, a^{-1}ca = c^{r}, b^{-1}db = d^{r}, r^{q} \equiv 1 \Mod{p}$, where $q$ divides $p-1$,
	\item[36.] $\langle a,b,c,d \mid a^{q} = b^{q} = c^{p} = d^{p} = [a,b] = [a,d] = [a,c] = [c,d] = 1, b^{-1}cb = c^{u}d^{vD}, b^{-1}db = c^{v}d^{u}$, where $u + v\sqrt{D} = \sigma^{\frac{p^{2}-1}{q}}$, $\sigma$ is a primitive root of Galois field $GF(p^{2})$, $u, v, D\in GF(p)$, $v \ne 0$, $D$ is not a perfect square, $q$ divides $p+1$ and $p\not\equiv 1 \Mod{q}$.
	
\end{enumerate}
 	\end{itemize}
 \end{theorem}


\section{Automorohism Groups of Groups of order $p^{2}q^{2}$}
In this section, we will compute the automorphism groups of the groups mentioned in the Theorem \ref{t1}. First, we will find the automorphism group of all groups of order $p^{2}q^{2}$, where $pq = 6$. Using GAP \cite{gap}, we have found the automorphism groups for cases $pq = 6$, which is listed in the following table.
\begin{table}[h]
\begin{tabular}{|c|c|c|}
\hline
Type & 	Group Description $(G)$ & Structure of $Aut(G)$\\
		\hline
	\hline
1&	$\mathbb{Z}_{36}$ & $U_{36}\simeq \mathbb{Z}_{6}\times \mathbb{Z}_{2}$\\
	\hline
2&	$\mathbb{Z}_{18} \times \mathbb{Z}_{2}$ & $\mathbb{Z}_{6}\times S_{3}$\\
	\hline
3&	$\mathbb{Z}_{6}\times \mathbb{Z}_{6}$ & $S_{3}\times Gl(2,3)$\\
	\hline
4&	$\mathbb{Z}_{12}\times \mathbb{Z}_{3}$ & $\mathbb{Z}_{2}\times GL(2,3)$\\
	\hline
5&	$\mathbb{Z}_{6}\times D_{3}$ & $\mathbb{Z}_{2}\times \mathbb{Z}_{2}\times S_{3}$\\
	\hline
6&	$D_{3}\times D_{3}$ & $(S_{3}\times S_{3})\rtimes \mathbb{Z}_{2}$\\
	\hline 
7&	$D_{9}\times \mathbb{Z}_{2}$ & $\mathbb{Z}_{2}\times (\mathbb{Z}_{9}\rtimes \mathbb{Z}_{6})$\\
	\hline
8&	$A_{4}\times \mathbb{Z}_{3}$ & $S_{4}\times S_{3}$\\
	\hline
9&	$D_{18}$ & $\mathbb{Z}_{2}\times (\mathbb{Z}_{9}\rtimes \mathbb{Z}_{6})$\\
	\hline
10&	$H\times \mathbb{Z}_{3} \simeq (\mathbb{Z}_{3}\rtimes \mathbb{Z}_{4})\times \mathbb{Z}_{3}$ & $\mathbb{Z}_{2}\times \mathbb{Z}_{3}\times S_{3}$\\
	\hline
11&	$K\times \mathbb{Z}_{2}\simeq ((\mathbb{Z}_{3}\times \mathbb{Z}_{3})\rtimes \mathbb{Z}_{2})\times \mathbb{Z}_{2}$ & $\mathbb{Z}_{2}\times ((((\mathbb{Z}_{3}\times \mathbb{Z}_{3})\rtimes Q_{8}) \rtimes \mathbb{Z}_{3}) \rtimes \mathbb{Z}_{2})$\\
	\hline
12&	$(\mathbb{Z}_{3}\times \mathbb{Z}_{3})\rtimes \mathbb{Z}_{4}$ & $(\mathbb{Z}_{3}\times \mathbb{Z}_{3})\rtimes (\mathbb{Z}_{8}\rtimes \mathbb{Z}_{2})$\\
	\hline
13&	$(\mathbb{Z}_{2}\times \mathbb{Z}_{2})\rtimes \mathbb{Z}_{9}$ & $\mathbb{Z}_{3}\times S_{4}$\\
	\hline
14&	$(\mathbb{Z}_{3}\times \mathbb{Z}_{3})\rtimes \mathbb{Z}_{4}$ & $(\mathbb{Z}_{3}\times \mathbb{Z}_{3})\rtimes (GL(2,3)\times \mathbb{Z}_{2})$\\
	 \hline	
\end{tabular}
\caption{ Structure of $Aut(G)$ of all groups $G$ of order $p^{2}q^{2}$ such that $pq=6.$}
\end{table}

Now, we will compute the automorphism group $Aut(G)$ of all groups of order $p^{2}q^{2}$ upto isomorphism in the Theorem \ref{t1} such that $pq\ne 6$. Note that, in the following cases, we will take $H$ to be isomorphic to either $\mathbb{Z}_{p}\times \mathbb{Z}_{p}$ or $\mathbb{Z}_{p^{2}}$ and $K$ to be isomorphic to either $\mathbb{Z}_{q}\times \mathbb{Z}_{q}$ or $\mathbb{Z}_{q^{2}}$. Since $p$ and $q$ are distinct primes, $Hom(H, K), Hom(K,H)$ and $Hom(H/[H,K], K)$ are all trivial. Therefore, the map $\gamma = 0$ always.

\subsection{Type 15.} Let $G \simeq \mathbb{Z}_{p^{2}}\times \mathbb{Z}_{q^{2}}$. Since, $p$ and $q$ are distinct primes, $\gcd(p,q) = 1$. Therefore, $G \simeq \mathbb{Z}_{p^{2}q^{2}}$. Thus, $Aut(G) \simeq  \mathbb{Z}_{p(p-1)}\times \mathbb{Z}_{q(q-1)}$. \label{g15}
	
\subsection{Type 16.} Let $G \simeq \mathbb{Z}_{q}\times \mathbb{Z}_{q}\times \mathbb{Z}_{p^{2}}$. Let $H \simeq  \mathbb{Z}_{q}\times \mathbb{Z}_{q}$ and $K \simeq \mathbb{Z}_{p^{2}}$. Then $A \simeq Aut(H)\simeq GL(2,q)$ and $D \simeq Aut(K)\simeq \mathbb{Z}_{p(p-1)}$. Hence, using the Theorem \ref{dp}, $Aut(G) \simeq A\times D \simeq GL(2,q)\times \mathbb{Z}_{p(p-1)}$. \label{g16}
	
	\subsection{Type 17.} Let $G \simeq \mathbb{Z}_{q^{2}}\times \mathbb{Z}_{p}\times \mathbb{Z}_{p}$. Then, using the Theorem \ref{dp}, $Aut(G) \simeq \mathbb{Z}_{q(q-1)}\times GL(2,p)$. \label{g17}
	
\subsection{Type 18.} Let $G \simeq \mathbb{Z}_{q}\times \mathbb{Z}_{q}\times \mathbb{Z}_{p}\times \mathbb{Z}_{p}$. Let $H \simeq \mathbb{Z}_{q}\times \mathbb{Z}_{q}$ and $K\simeq \mathbb{Z}_{p}\times \mathbb{Z}_{p}$. Then using the similar argument as in \ref{g16}, we get $A\simeq Aut(H) \simeq GL(2,q)$ and $D \simeq Aut(K)\simeq GL(2,p)$. Hence, $Aut(G) \simeq A\times D \simeq GL(2,q)\times GL(2,p)$. \label{g18}
	
	\subsection{Type 19.}  Let $K = \langle a\mid a^{q^{2}}=1\rangle \simeq \mathbb{Z}_{q^{2}}$, $H = \langle b\mid b^{p^{2}}=1\rangle \simeq \mathbb{Z}_{p^{2}}$ and $\phi: K \longrightarrow Aut(H)$ be the homomorphism defined by $\phi(a)(b) = a^{-1}ba = b^{r}$, where $r^{q}\equiv 1 \Mod{p^{2}}$. Then $G \simeq H\rtimes_{\phi} K$. Now, let $(\alpha, \delta) \in P$. Then $\alpha(b) = b^{i}$ and $\delta(a) = a^{j}$, where $1\le i\le  p^{2}-1, \gcd(i,p) =1$ and $1\le j\le q^{2}-1, \gcd(j,q) =1$. Since $(\alpha, \delta) \in P$, $\alpha(b^{a}) = \alpha(b)^{\delta(a)}$. Now, $b^{ri} = \alpha(b^{r}) = \alpha(b^{a}) = \alpha(b)^{\delta(a)} = (b^{i})^{a^{j}} = b^{ir^{j}}$. Thus, $ri\equiv ir^{j}\Mod{p^{2}}$, which implies that $r^{j-1}\equiv 1\Mod{p^{2}}$. Therefore, $j\equiv 1 \Mod{q}$ and so, $Aut(K)\simeq \mathbb{Z}_{q}$. Therefore, $R\simeq \mathbb{Z}_{q}\times \mathbb{Z}_{p(p-1)}$.
		
	Now, let $\beta\in S$ be defined by $\beta(a) = b^{\lambda}$, where $0\le \lambda\le p^{2}-1$. Then $\beta(a^{2}) = \beta(a)\beta(a)^{a} = b^{\lambda}(b^{\lambda})^{a} = b^{\lambda(1+r)}$. Inductively, we get, $\beta(a^{l}) = b^{\lambda(1+r+r^{2}+\cdot + r^{l-1})} = b^{\lambda\frac{r^{l}-1}{r-1}}$ for any $0\le l\le q^{2}-1$. In particular, $\beta(a^{q}) = \beta(a^{q^{2}}) = 1$. Thus, $S\simeq Q \simeq \mathbb{Z}_{p^{2}}$. Hence, using the Theorem \ref{main}, $Aut(G) \simeq Q\rtimes R \simeq \mathbb{Z}_{p^{2}}\rtimes (\mathbb{Z}_{p(p-1)}\times \mathbb{Z}_{q})$. \label{g19}
	
	\subsection{Type 20.}  Let $K = \langle a\mid a^{q^{2}}=1\rangle \simeq \mathbb{Z}_{q^{2}}$, $H = \langle b\mid b^{p^{2}}=1\rangle \simeq \mathbb{Z}_{p^{2}}$ and $\phi: K \longrightarrow Aut(H)$ be the homomorphism defined by $\phi(a)(b) = a^{-1}ba = b^{r}$, where $r^{q^{2}}\equiv 1 \Mod{p^{2}}$. Then $G\simeq H\rtimes_{\phi} K$. Using the similar argument as in the above part \ref{g19}, we get $R \simeq \mathbb{Z}_{p(p-1)}$ and $Q \simeq \mathbb{Z}_{p^{2}}$. Hence, using the Theorem \ref{main}, $Aut(G) \simeq Q\rtimes R \simeq \mathbb{Z}_{p^{2}}\rtimes \mathbb{Z}_{p(p-1)}$. \label{g20}
	
	\subsection{Type 21.} Let $H = \langle c\mid c^{p^{2}}=1\rangle \simeq \mathbb{Z}_{p^{2}}$, $K = \langle a,b \mid a^{q} = b^{q} = [a,b] = 1\rangle \simeq \mathbb{Z}_{q}\times \mathbb{Z}_{q}$ and the homomorphism $\phi: K \longrightarrow Aut(H)$ be defined as $\phi(a)(c) = a^{-1}ca = c^{r}, \phi(b)(c) = b^{-1}cb = c$, where $r^{q}\equiv 1 \Mod{p^{2}}$. Then $G\simeq H\rtimes_{\phi} K$. Now, let $(\alpha, \delta)\in P$. Then $\alpha\in Aut(H)$ is defined by $\alpha(c) = c^{s}$ and $\delta\in Aut(K) \simeq GL(2,q)$ is given by $\delta = \begin{pmatrix}
	i& l\\ j & k
	\end{pmatrix}$ such that $\delta(a) = a^{i}b^{j}, \delta(b) = a^{l}b^{k}$, where $1\le s \le p^{2}-1, \gcd(s,p) = 1$, $0\le i,j,l,k \le q-1$ and $ik\not\equiv jl \Mod{q}$. Since $(\alpha, \delta)\in P$,  $\alpha(c^{a}) = \alpha(c)^{\delta(a)}$ and $\alpha(c^{b}) = \alpha(c)^{\delta(b)}$.
	
	Now, $c^{rs} = \alpha(c^{a}) = \alpha(c)^{\delta(a)} = (c^{s})^{a^{i}b^{j}} = c^{sr^{i}}$. Then $rs \equiv sr^{i} \Mod{p^{2}}$ which implies that $r^{i-1}\equiv 1 \Mod{p^{2}}$. So, $i\equiv 1 \Mod{q}$. Also, $c^{s} = \alpha(c^{b}) = \alpha(c)^{\delta(b)} = (c^{s})^{a^{l}b^{k}} = c^{sr^{l}}$. Then $sr^{l} \equiv s \Mod{p^{2}}$ which implies that $r^{l} \equiv 1\Mod{p^{2}}$. Therefore, $l \equiv 0\Mod{q}$. Thus $P\sim R \simeq \mathbb{Z}_{p(p-1)}\times (\mathbb{Z}_{q}\rtimes \mathbb{Z}_{q-1})$. Let $\beta\in S$ be defined by $\beta(a) = c^{\lambda}$ and $\beta(b) = c^{\rho}$, where $0\le \lambda, \rho \le p^{2}-1$. Now, $\beta(a^{2}) = \beta(a)\beta(a)^{a} = c^{\lambda}(c^{\lambda})^{a} = c^{\lambda(1+r)}$. Inductively, we get for $0\le l\le q-1$, $\beta(a^{l}) = c^{\lambda(1+r+r^{2}+\cdots + r^{l-1})} = c^{\lambda\frac{r^{l}-1}{r-1}}$. In particular, $\beta(a^{q}) = c^{\lambda\frac{r^{q}-1}{r-1}} = 1$. Also, $\beta(b^{2}) = c^{2\rho}$ and $ \beta(b^{q}) = c^{q\rho}$. Since $\beta(b^{q}) = 1$, $q\rho \equiv 0\Mod{p^{2}}$ which implies that $\rho = 0$. Thus, we have $\beta(a) = c^{\lambda}$ and $\beta(b) = 1$, where $0\le \lambda \le p^{2}-1$ and so, $S \simeq Q\simeq \mathbb{Z}_{p^{2}}$. Hence, using the Theorem \ref{main}, $Aut(G) \simeq Q\rtimes R \simeq \mathbb{Z}_{p^{2}}\rtimes (\mathbb{Z}_{p(p-1)}\times (\mathbb{Z}_{q}\rtimes \mathbb{Z}_{q-1}))$. \label{g21}
	
		\subsection{Type 22.}  Let $H = \langle b, c\mid b^{p} = c^{p}=1, bc=cb\rangle \simeq \mathbb{Z}_{p}\times \mathbb{Z}_{p}$, $K = \langle a \mid a^{q^{2}} =1\rangle \simeq \mathbb{Z}_{q^{2}}$ and the homomorphism $\phi_{r}: K \longrightarrow Aut(H)$ be defined as $\phi_{r}(a)(b) = a^{-1}ba = b^{r}, \phi_{r}(a)(c) = a^{-1}ca = c^{r}$, where $r^{q}\equiv 1 \Mod{p^{2}}$. Then $G\simeq H\rtimes_{\phi_{r}} K$. Now, let $(\alpha, \delta)\in P$. Then $\delta\in Aut(K)$ is defined by $\delta(a) = a^{s}$ and $\alpha\in Aut(H) \simeq GL(2,p)$ is given by $\alpha = \begin{pmatrix}
	i& l\\ j & k
	\end{pmatrix}$ such that $\alpha(b) = b^{i}c^{j}, \alpha(c) = b^{l}c^{k}$, where $1\le s \le q^{2}-1, \gcd(s,q) = 1$, $0\le i,j,l,k \le p-1$ and $ik\not\equiv jl \Mod{p}$. Since, $(\alpha, \delta)\in P$, $\alpha(b^{a}) = \alpha(b)^{\delta(a)}$ and $\alpha(c^{a}) = \alpha(c)^{\delta(a)}$.
	
	Now, $b^{ir}c^{jr} = \alpha(b^{r}) = \alpha(b^{a}) = \alpha(b)^{\delta(a)} = (b^{i}c^{j})^{a^{s}} = b^{ir^{s}}c^{jr^{s}}$. Then $ir \equiv ir^{s}\Mod{p}$ and $jr \equiv jr^{s}\Mod{p}$ which implies that $r^{s-1}\equiv 1\Mod{p}$ and so, $s\equiv 1\Mod{q}$. Also, using $\alpha(c^{a}) = \alpha(c)^{\delta(a)}$, we get $s\equiv 1\Mod{q}$. So, we get $\alpha = \begin{pmatrix}
	i& l\\ j & k
	\end{pmatrix}$ and $\delta(a) = a^{s}$, where $s\equiv 1\Mod{q}$, $0\le i,j,l,k \le p-1$ and $ik\not\equiv jl \Mod{p}$. Thus $P\simeq R\simeq GL(2,p)\times \mathbb{Z}_{q}$. Now, let $\beta\in S$ be defined by $\beta(a) = b^{\lambda}c^{\rho}$, where $0\le \lambda, \rho \le p-1$. Then $\beta(a^{2}) = b^{\lambda(1+ r)}c^{\rho(1+r)}$ and $\beta(a^{q}) = b^{\lambda(1+ r + r^{2}+ \cdots + r^{q-1})}c^{\rho(1+ r + r^{2}+ \cdots + r^{q-1})} = b^{\lambda\frac{r^{q}-1}{r-1}}c^{\rho\frac{r^{q}-1}{r-1}} = 1 = \beta(a^{q^{2}})$. Thus $\beta(a) = b^{\lambda}c^{\rho}$, where $0\le \lambda, \rho \le p-1$ and so, $S\simeq Q\simeq \mathbb{Z}_{p}\times \mathbb{Z}_{p}$. Hence, using the Theorem \ref{main}, $Aut(G)\simeq (\mathbb{Z}_{p}\times \mathbb{Z}_{p})\rtimes (GL(2,p)\times \mathbb{Z}_{q})$. \label{g22}
	
		\subsection{Type 23.}  Let $H = \langle b, c\mid b^{p} = c^{p}=1, bc=cb\rangle \simeq \mathbb{Z}_{p}\times \mathbb{Z}_{p}$, $K = \langle a \mid a^{4} =1\rangle \simeq \mathbb{Z}_{4}$ and the homomorphism $\phi: K \longrightarrow Aut(H)$ be defined as $\phi(a)(b) = a^{-1}ba = b, \phi(a)(c) = a^{-1}ca = c^{-1}$. Then $G\simeq H\rtimes_{\phi} K$. Now, let $(\alpha, \delta)\in P$. Then $\delta\in Aut(K)$ is defined by $\delta(a) = a^{s}$ and $\alpha\in Aut(H) \simeq GL(2,p)$ is given by $\alpha = \begin{pmatrix}
	i& l\\ j & k
	\end{pmatrix}$ such that $\alpha(b) = b^{i}c^{j}, \alpha(c) = b^{l}c^{k}$, where $s\in \{1,3\}$, $0\le i,j,l,k \le p-1$ and $ik\not\equiv jl \Mod{p}$. Since, $(\alpha, \delta)\in P$, $\alpha(b^{a}) = \alpha(b)^{\delta(a)}$ and $\alpha(c^{a}) = \alpha(c)^{\delta(a)}$.
	
	Now, $b^{i}c^{j} = \alpha(b) = \alpha(b^{a}) = \alpha(b)^{\delta(a)} = (b^{i}c^{j})^{a^{s}} = b^{i}c^{-j}$. Then $j\equiv -j \Mod{p}$ which implies that $j = 0$. Also, $b^{-l}c^{-k} = \alpha(c^{-1}) = \alpha(c^{a}) = \alpha(c)^{\delta(a)} = (b^{l}c^{k})^{a^{s}} = b^{l}c^{-k}$. Then $-l \equiv l \Mod{p}$ which implies that $l = 0$. Thus, $\alpha = \begin{pmatrix}
	i& 0\\ 0 & k
	\end{pmatrix}$ and $\delta(a) = a^{s}$, where $1\le i,k \le p-1$ and $s\in \{1,3\}$ and so, $R\simeq (\mathbb{Z}_{p-1}\times \mathbb{Z}_{p-1})\times \mathbb{Z}_{2}$. Now, let $\beta\in S$ be defined by $\beta(a) = b^{\lambda}c^{\rho}$, where $0\le \lambda, \rho \le p-1$. Then $\beta(a^{2}) = b^{2\lambda}$, $\beta(a^{3}) = b^{3\lambda}c^{\rho}$  and $1 = \beta(a^{4}) = b^{4\lambda}$. Thus $4\lambda \equiv0\Mod{p}$ which implies that $\lambda = 0$ and so, $\beta(a) = c^{\rho}$, where $0\le \rho \le p-1$ and so, $Q\simeq \mathbb{Z}_{p}$. Hence, using the Theorem \ref{main}, $Aut(G)\simeq \mathbb{Z}_{p}\rtimes ((\mathbb{Z}_{p-1}\times \mathbb{Z}_{p-1})\times \mathbb{Z}_{2})$. \label{g23}
	
		\subsection{Type 24.}  Let $H = \langle b, c\mid b^{p} = c^{p}=1, bc=cb\rangle \simeq \mathbb{Z}_{p}\times \mathbb{Z}_{p}$, $K = \langle a \mid a^{q^{2}} =1\rangle \simeq \mathbb{Z}_{q^{2}}$ and the homomorphism $\phi: K \longrightarrow Aut(H)$ be defined as $\phi(a)(b) = a^{-1}ba = b, \phi(a)(c) = a^{-1}ca = c^{r}$, where $r^{q}\equiv 1 \Mod{p}$. Then $G\simeq H\rtimes_{\phi} K$. Now, let $(\alpha, \delta)\in P$. Then $\delta\in Aut(K)$ is defined by $\delta(a) = a^{s}$ and $\alpha\in Aut(H) \simeq GL(2,p)$ is given by $\alpha = \begin{pmatrix}
	i& l\\ j & k
	\end{pmatrix}$ such that $\alpha(b) = b^{i}c^{j}, \alpha(c) = b^{l}c^{k}$, where $1\le s \le q^{2}-1, \gcd(s,q) = 1$, $0\le i,j,l,k \le p-1$ and $ik\not\equiv jl \Mod{p}$. Since, $(\alpha, \delta)\in P$, $\alpha(b^{a}) = \alpha(b)^{\delta(a)}$ and $\alpha(c^{a}) = \alpha(c)^{\delta(a)}$.
	
	Now, $b^{i}c^{j} = \alpha(b) = \alpha(b^{a}) = \alpha(b)^{\delta(a)} = (b^{i}c^{j})^{a^{s}} = b^{i}c^{jr^{s}}$. Then $j\equiv jr^{s}\Mod{p}$. Since, $\gcd(s,q) = 1$, $r^{s}\not\equiv 1 \Mod{p}$. Therefore, $j = 0$. Also, $b^{rl}c^{rk} = \alpha(c^{r}) = \alpha(c^{a}) = \alpha(c)^{\delta(a)} = (b^{l}c^{k})^{a^{s}} = b^{l}c^{kr^{s}}$. Then, $rl \equiv l\Mod{p}$
 and $rk\equiv kr^{s}\Mod{p}$ which implies that $l = 0$ and $(r^{s-1}-1)k\equiv 0\Mod{p}$. So, $k\ne 0$, otherwise $\alpha\not\in GL(2,p)$. Therefore, $r^{s-1} \equiv 1 \Mod{p}$ and so, $s\equiv 1 \Mod{q}$. Thus, we have $\alpha = \begin{pmatrix}
 i& 0\\ 0 & k
 \end{pmatrix}$ and $\delta(a) = a^{s}$, where $s\equiv 1 \Mod{q}$ and $1\le i,k \le p-1$ and so, $R \simeq (\mathbb{Z}_{p-1}\times \mathbb{Z}_{p-1})\times \mathbb{Z}_{q}$. Now, let $\beta\in S$ be defined by $\beta(a) = b^{\lambda}c^{\rho}$, where $0\le \lambda, \rho \le p-1$. Then $\beta(a^{2}) = b^{2\lambda}c^{\rho(1+r)}$ and for any $0\le l\le q^{2}-1$, $\beta(a^{l}) = b^{l\lambda}c^{\rho\frac{r^{l}-1}{r-1}}$. In particular, $1 = \beta(a^{q^{2}}) = b^{q^{2}\lambda}$. Thus $q^{2}\lambda \equiv0\Mod{p}$ which implies that $\lambda = 0$. Therefore, $\beta(a) = c^{\rho}$, where $0\le \rho \le p-1$ and so, $Q\simeq \mathbb{Z}_{p}$. Hence, using the Theorem \ref{main}, $Aut(G)\simeq \mathbb{Z}_{p}\rtimes ((\mathbb{Z}_{p-1}\times \mathbb{Z}_{p-1})\times \mathbb{Z}_{q})$. \label{g24}
 
 	\subsection{Type 25.}  Let $H = \langle b, c\mid b^{p} = c^{p}=1, bc=cb\rangle \simeq \mathbb{Z}_{p}\times \mathbb{Z}_{p}$, $K = \langle a \mid a^{q^{2}} =1\rangle \simeq \mathbb{Z}_{q^{2}}$ and the homomorphism $\phi_{r^{-1}}: K \longrightarrow Aut(H)$ be defined as $\phi_{r^{-1}}(a)(b) = a^{-1}ba = b^{r}, \phi_{r^{-1}}(a)(c) = a^{-1}ca = c^{r^{-1}}$, where $r^{q}\equiv 1 \Mod{p}$. Then $G\simeq H\rtimes_{\phi_{r^{-1}}} K$. Now, let $(\alpha, \delta)\in P$. Then $\delta\in Aut(K)$ is defined by $\delta(a) = a^{s}$ and $\alpha\in Aut(H) \simeq GL(2,p)$ is given by $\alpha = \begin{pmatrix}
 i& l\\ j & k
 \end{pmatrix}$ such that $\alpha(b) = b^{i}c^{j}, \alpha(c) = b^{l}c^{k}$, where $1\le s \le q^{2}-1, \gcd(s,q) = 1$, $0\le i,j,l,k \le p-1$ and $ik\not\equiv jl \Mod{p}$. Since, $(\alpha, \delta)\in P$, $\alpha(b^{a}) = \alpha(b)^{\delta(a)}$ and $\alpha(c^{a}) = \alpha(c)^{\delta(a)}$.
 
 Now, $b^{ri}c^{rj} = \alpha(b^{r}) = \alpha(b^{a}) = \alpha(b)^{\delta(a)} = (b^{i}c^{j})^{a^{s}} = b^{ir^{s}}c^{jr^{-s}}$ and $b^{lr^{-1}}c^{kr^{-1}} = \alpha(c^{r^{-1}}) = \alpha(c^{a}) = \alpha(c)^{\delta(a)} = (b^{l}c^{k})^{a^{s}} = b^{lr^{s}}c^{kr^{-s}}$. Then $ri \equiv ir^{s}\Mod{p}$, $rj \equiv jr^{-s}\Mod{p}$, $lr^{-1}\equiv lr^{s}\Mod{p}$ and $kr^{-1}\equiv kr^{-s} \Mod{p}$. Since both $i, j$ can not vanish together. We have two cases, namely either $i\not\equiv 0\Mod{p}$ or $j\not\equiv 0\Mod{p}$. Also, we will observe that both $i$ and $j$ can not be non-zero at the same time.
 
 First, let $i\not\equiv 0\Mod{p}$. Then $r^{s-1}\equiv 1\Mod{p}$ which implies that $s\equiv 1 \Mod{q}$. Also, using $rj \equiv jr^{-s}\Mod{p}$, we get $(r^{2}-1)j\equiv 0\Mod{p}$ which implies that $j = 0$ as $r^{2}\not\equiv 1\Mod{p}$. Similarly, $l = 0$. So, in this case $\alpha = \begin{pmatrix}
 i& 0\\ 0 & k
 \end{pmatrix}$ and $\delta(a) = a^{s}$, where $s\equiv 1 \Mod{q}$ and $1\le i,k \le p-1$. Now, let $j\not\equiv 0\Mod{p}$. Then $r^{s+1}\equiv 1\Mod{p}$ which implies that $s\equiv -1\Mod{q}$. Also, using $ri \equiv ir^{s}\Mod{p}$, we get $(r^{2}-1)i\equiv 0\Mod{p}$ which implies that $i = 0$ as $r^{2}\not\equiv 1\Mod{p}$. Similarly, $k = 0$. So, in this case $\alpha = \begin{pmatrix}
 0& j\\ l & 0
 \end{pmatrix}$ and $\delta(a) = a^{s}$, where $s\equiv -1 \Mod{q}$ and $1\le i,k \le p-1$. From both the cases, we have $R\simeq \langle \left(\begin{pmatrix}
 i& 0\\ 0 & k
 \end{pmatrix},\delta_{s}\right), \left(\begin{pmatrix}
 0& j\\ l & 0
 \end{pmatrix}, \delta_{-s}\right) \rangle \simeq ((\mathbb{Z}_{p-1}\times \mathbb{Z}_{p-1})\times \mathbb{Z}_{q})\rtimes \mathbb{Z}_{2}$, where
   $\delta_{s}(a) = a^{s}$ and $s\equiv \pm 1\Mod{p}$. 
 
	Now, let $\beta\in S$ be defined by $\beta(a) = b^{\lambda}c^{\rho}$, where $0\le \lambda, \rho \le p-1$. Then $\beta(a^{2}) = b^{\lambda(1+r)}c^{\rho(1+r^{-1})}$ and for any $0\le l\le q^{2}-1$, $\beta(a^{l}) = b^{\lambda\frac{r^{l}-1}{r-1}}c^{\rho\frac{r^{-l}-1}{r-1}}$. In particular, $\beta(a^{q^{2}}) = b^{\lambda\frac{r^{q^{2}}-1}{r-1}}c^{\rho\frac{r^{-q^{2}}-1}{r-1}} = 1$. Thus $\beta(a) = b^{\lambda}c^{\rho}$, where $0\le \lambda, \rho \le p-1$ and so, $Q\simeq \mathbb{Z}_{p}\times \mathbb{Z}_{p}$. Hence, using the Theorem \ref{main}, $Aut(G)\simeq (\mathbb{Z}_{p}\times \mathbb{Z}_{p})\rtimes (((\mathbb{Z}_{p-1}\times \mathbb{Z}_{p-1})\times \mathbb{Z}_{q})\rtimes \mathbb{Z}_{2})$. \label{g25}
	
		\subsection{Type 26.} Let $H = \langle b, c\mid b^{p} = c^{p}=1, bc=cb\rangle \simeq \mathbb{Z}_{p}\times \mathbb{Z}_{p}$, $K = \langle a \mid a^{q^{2}} =1\rangle \simeq \mathbb{Z}_{q^{2}}$ and the homomorphism $\phi_{r^{-1}}: K \longrightarrow Aut(H)$ be defined as $\phi_{r^{-1}}(a)(b) = a^{-1}ba = b^{r}, \phi_{r^{-1}}(a)(c) = a^{-1}ca = c^{r^{-1}}$, where $r^{q^{2}}\equiv 1 \Mod{p}$. Then $G\simeq H\rtimes_{\phi_{r^{-1}}} K$. Using the similar argument as in  \ref{g25}, we get $R\simeq (\mathbb{Z}_{p-1}\times \mathbb{Z}_{p-1})\rtimes \mathbb{Z}_{2}$ and $Q\simeq \mathbb{Z}_{p}\times \mathbb{Z}_{p}$. Hence, using the Theorem \ref{main}, $Aut(G)\simeq (\mathbb{Z}_{p}\times \mathbb{Z}_{p})\rtimes ((\mathbb{Z}_{p-1}\times \mathbb{Z}_{p-1})\rtimes \mathbb{Z}_{2})$. \label{g26}
	
		\subsection{Type 27.}  Let $H = \langle b, c\mid b^{p} = c^{p}=1, bc=cb\rangle \simeq \mathbb{Z}_{p}\times \mathbb{Z}_{p}$, $K = \langle a \mid a^{q^{2}} =1\rangle \simeq \mathbb{Z}_{q^{2}}$ and the homomorphism $\phi_{r^{n}}: K \longrightarrow Aut(H)$ be defined as $\phi_{r^{n}}(a)(b) = a^{-1}ba = b^{r}, \phi_{r^{n}}(a)(c) = a^{-1}ca = c^{r^{n}}$, where $r^{q}\equiv 1 \Mod{p}$ and $n\in \{2,3 \cdots, \frac{q-1}{2}\}$. Then $G\simeq H\rtimes_{\phi_{r^{n}}} K$. Now, let $(\alpha, \delta)\in P$. Then $\delta\in Aut(K)$ is defined by $\delta(a) = a^{s}$ and $\alpha\in Aut(H) \simeq GL(2,p)$ is given by $\alpha = \begin{pmatrix}
	i& l\\ j & k
	\end{pmatrix}$ such that $\alpha(b) = b^{i}c^{j}, \alpha(c) = b^{l}c^{k}$, where $1\le s \le q^{2}-1, \gcd(s,q) = 1$, $0\le i,j,l,k \le p-1$ and $ik\not\equiv jl \Mod{p}$. Since, $(\alpha, \delta)\in P$, $\alpha(b^{a}) = \alpha(b)^{\delta(a)}$ and $\alpha(c^{a}) = \alpha(c)^{\delta(a)}$.
	
	Now, $b^{ri}c^{rj} = \alpha(b^{r}) = \alpha(b^{a}) = \alpha(b)^{\delta(a)} = (b^{i}c^{j})^{a^{s}} = b^{ir^{s}}c^{jr^{ns}}$ and $b^{lr^{n}}c^{kr^{n}} = \alpha(c^{r^{n}}) = \alpha(c^{a}) = \alpha(c)^{\delta(a)} = (b^{l}c^{k})^{a^{s}} = b^{lr^{s}}c^{kr^{ns}}$. Then $ri \equiv ir^{s}\Mod{p}$, $rj \equiv jr^{ns}\Mod{p}$, $lr^{n}\equiv lr^{s}\Mod{p}$ and $kr^{n}\equiv kr^{ns} \Mod{p}$. Since both $i, j$ can not vanish together. We have two cases, namely either $i\not\equiv 0\Mod{p}$ or $j\not\equiv 0\Mod{p}$. Also, we will observe that both $i$ and $j$ can not be non-zero at the same time.
	
	First, let $i\not\equiv 0\Mod{p}$. Then $r^{s-1}\equiv 1\Mod{p}$ which implies that $s\equiv 1 \Mod{q}$. Also, using $rj \equiv jr^{ns}\Mod{p}$, we get $(r^{n-1}-1)j\equiv 0\Mod{p}$. If $j\not\equiv 0\Mod{p}$, then $r^{n-1}\equiv 1 \Mod{p}$ and so, $n\equiv 1 \Mod{q}$ which is a contradiction. Therefore, $j = 0$. Similarly, $l = 0$. So, in this case $\alpha = \begin{pmatrix}
	i& 0\\ 0 & k
	\end{pmatrix}$ and $\delta(a) = a^{s}$, where $s\equiv 1 \Mod{q}$ and $1\le i,k \le p-1$. 
	
	Now, assume that $j\not\equiv 0\Mod{p}$. Then $r^{ns-1}\equiv 1\Mod{p}$ which implies that $ns\equiv 1\Mod{q}$. Using $kr^{n}\equiv kr^{ns} \Mod{p}$, we get $k\equiv 0\Mod{p}$. Now, using $lr^{n}\equiv lr^{s}\Mod{p}$, $lr^{n^{2}}\equiv lr^{ns}\Mod{p}$ which implies that $l\equiv 0 \Mod{p}$, otherwise $n^{2}\equiv 1 \Mod{q}$ which is a contradiction.  But, then $\alpha\not\in GL(2,p)$. Therefore, the case $j\not\equiv 0\Mod{p}$ is not possible. Thus, we have $R \simeq (\mathbb{Z}_{p-1}\times \mathbb{Z}_{p-1})\times \mathbb{Z}_{q}$. Now, let $\beta\in S$ be defined by $\beta(a) = b^{\lambda}c^{\rho}$, where $0\le \lambda, \rho \le p-1$. Then for any $l (0\le l \le q^{2}-1)$, we have
	$\beta(a^{l}) = b^{\lambda \frac{r^{l}-1}{r-1}}c^{\rho \frac{r^{nl}-1}{r-1}}$. Clearly, $\beta(a^{q^{2}}) = 1$ and so, $Q\simeq \mathbb{Z}_{p}\times \mathbb{Z}_{p}$. Hence, using the Theorem \ref{main}, $Aut(G)\simeq (\mathbb{Z}_{p}\times \mathbb{Z}_{p})\rtimes ((\mathbb{Z}_{p-1}\times \mathbb{Z}_{p-1})\rtimes \mathbb{Z}_{q})$. \label{g27}
	
		\subsection{Type 28.}  Let $H = \langle b, c\mid b^{p} = c^{p}=1, bc=cb\rangle \simeq \mathbb{Z}_{p}\times \mathbb{Z}_{p}$, $K = \langle a \mid a^{q^{2}} =1\rangle \simeq \mathbb{Z}_{q^{2}}$ and the homomorphism $\phi_{r^{n}}: K \longrightarrow Aut(H)$ be defined as $\phi_{r^{n}}(a)(b) = a^{-1}ba = b^{r}, \phi_{r^{n}}(a)(c) = a^{-1}ca = c^{r^{n}}$, where $r^{q^{2}}\equiv 1 \Mod{p}$ and $2\le n\le \frac{q^{2}-1}{2}$ or $n = mq(m\ge \frac{q+1}{2})$. Then $G\simeq H\rtimes_{\phi_{r^{n}}} K$. Using the similar argument as in \ref{g27}, we get $R\simeq \mathbb{Z}_{p-1}\times \mathbb{Z}_{p-1}$ and $Q\simeq \mathbb{Z}_{p}\times \mathbb{Z}_{p}$. Hence, using the Theorem \ref{main}, $Aut(G)\simeq (\mathbb{Z}_{p}\times \mathbb{Z}_{p})\rtimes (\mathbb{Z}_{p-1}\times \mathbb{Z}_{p-1})$. \label{g28}

	\subsection{Type 29.}  Let $H = \langle b, c\mid b^{p} = c^{p}=1, bc=cb\rangle \simeq \mathbb{Z}_{p}\times \mathbb{Z}_{p}$, $K = \langle a \mid a^{q^{2}} =1\rangle \simeq \mathbb{Z}_{q^{2}}$ and the homomorphism $\phi: K \longrightarrow Aut(H)$ be defined as $\phi(a)(b) = a^{-1}ba = b^{r}, \phi(a)(c) = a^{-1}ca = c^{r}$, where $r^{q^{2}}\equiv 1 \Mod{p}$ and $n\in \{2,3, \cdots, \frac{q-1}{2}\}$. Then $G\simeq H\rtimes_{\phi} K$. Now, let $(\alpha, \delta)\in P$. Then $\delta\in Aut(K)$ is defined by $\delta(a) = a^{s}$ and $\alpha\in Aut(H) \simeq GL(2,p)$ is given by $\alpha = \begin{pmatrix}
i& l\\ j & k
\end{pmatrix}$ such that $\alpha(b) = b^{i}c^{j}, \alpha(c) = b^{l}c^{k}$, where $1\le s \le q^{2}-1, \gcd(s,q) = 1$, $0\le i,j,l,k \le p-1$ and $ik\not\equiv jl \Mod{p}$. Since, $(\alpha, \delta)\in P$, $\alpha(b^{a}) = \alpha(b)^{\delta(a)}$ and $\alpha(c^{a}) = \alpha(c)^{\delta(a)}$.

	Now, $b^{ri}c^{rj} = \alpha(b^{r}) = \alpha(b^{a}) = \alpha(b)^{\delta(a)} = (b^{i}c^{j})^{a^{s}} = b^{ir^{s}}c^{jr^{s}}$ and $b^{lr}c^{kr} = \alpha(c^{r}) = \alpha(c^{a}) = \alpha(c)^{\delta(a)} = (b^{l}c^{k})^{a^{s}} = b^{lr^{s}}c^{kr^{s}}$. Then $ri \equiv ir^{s}\Mod{p}$, $rj \equiv jr^{s}\Mod{p}$, $lr\equiv lr^{s}\Mod{p}$ and $kr\equiv kr^{s} \Mod{p}$. Since, all $i,j,l,k$ can not vanish together, take $i\not\equiv 0 \Mod{p}$. Then we get, $r^{s-1}\equiv 1 \Mod{p}$ which implies that $s\equiv 1 \Mod{q^{2}}$. Therefore, $\alpha = \begin{pmatrix}
	i& l\\ j & k
	\end{pmatrix}$ and $\delta(a) = a$, where $0\le i,j,l,k \le p-1$ and $ik\not\equiv jl \Mod{p}$. Thus $R\simeq GL(2,p)$.
	
	Now, let $\beta\in S$ be defined by $\beta(a) = b^{\lambda}c^{\rho}$, where $0\le \lambda, \rho \le p-1$. Then for any $l (0\le l \le q^{2}-1)$, we have
	$\beta(a^{l}) = b^{\lambda \frac{r^{l}-1}{r-1}}c^{\rho \frac{r^{l}-1}{r-1}}$. Clearly, $\beta(a^{q^{2}}) = 1$ and so, $Q\simeq \mathbb{Z}_{p}\times \mathbb{Z}_{p}$. Hence, using the Theorem \ref{main}, $Aut(G)\simeq (\mathbb{Z}_{p}\times \mathbb{Z}_{p})\rtimes GL(2,p)$. \label{g29}
	
		\subsection{Type 30.}  Let $H = \langle b, c\mid b^{p} = c^{p}=1, bc=cb\rangle \simeq \mathbb{Z}_{p}\times \mathbb{Z}_{p}$, $K = \langle a \mid a^{q^{2}} =1\rangle \simeq \mathbb{Z}_{q^{2}}$ and the homomorphism $\phi: K \longrightarrow Aut(H)$ be defined as $\phi(a)(b) = a^{-1}ba = b^{m}c^{nD}, \phi(a)(c) = a^{-1}ca = b^{n}c^{m}$, where $m + n\sqrt{D} = \sigma^{\frac{p^{2}-1}{q}}$, $\sigma$ is a primitive root of Galois Field $GF(p^{2})$, $m, n, D\in GF(p)$, $n \ne 0$, $D$ is not a perfect square and $q$ divides $p+1$. Then $G\simeq H\rtimes_{\phi} K$. Now, let $(\alpha, \delta)\in P$. Then $\delta\in Aut(K)$ is defined by $\delta(a) = a^{s}$ and $\alpha\in Aut(H) \simeq GL(2,p)$ is given by $\alpha = \begin{pmatrix}
	i& l\\ j & k
	\end{pmatrix}$ such that $\alpha(b) = b^{i}c^{j}, \alpha(c) = b^{l}c^{k}$, where $1\le s \le q^{2}-1, \gcd(s,q) = 1$, $0\le i,j,l,k \le p-1$ and $ik\not\equiv jl \Mod{p}$. Since, $(\alpha, \delta)\in P$, $\alpha(b^{a}) = \alpha(b)^{\delta(a)}$ and $\alpha(c^{a}) = \alpha(c)^{\delta(a)}$.
	
	Now, $b^{mi+lnD}c^{jm+knD} = \alpha(b^{m}c^{nD}) = \alpha(b^{a}) = \alpha(b)^{\delta(a)} = (b^{i}c^{j})^{a^{s}} = b^{iM+ jN}c^{iDN + jM}$, and  $b^{in+lm}c^{jn+km} = \alpha(b^{n}c^{m}) = \alpha(c^{a}) = \alpha(c)^{\delta(a)} = (b^{l}c^{k})^{a^{s}} = b^{lM+ kN}c^{lDN + kM}$, where $M = \sum_{t=0}^{[\frac{n}{2}]}{{^{s}}C_{2t}m^{s-2t}n^{2t}D^{t}}$ and $N = \sum_{t=0}^{[\frac{n}{2}]}{^{s}C_{2t+1}m^{s-2t-1}n^{2t+1}D^{t}}$. Then
	\begin{align}
	mi+lnD &\equiv iM + jN \Mod{p}\label{e1}\\
	jm+knD &\equiv iDN+ jM\Mod{p} \label{e2}\\
	in+lm &\equiv lM+ kN \Mod{p} \label{e3}\\
	jn+km&\equiv lDN + kM \Mod{p} \label{e4}.	
	\end{align}
	Note that $(m+n\sqrt{D})^{s} = M + \sqrt{D}N$ and so, $M = (m+n\sqrt{D})^{s}- \sqrt{D}N$. Substituting the value of $M$ in the Congruence relations (\ref{e1})-(\ref{e4}), we get   
	\begin{align}
	mi+lnD &\equiv i(m+n\sqrt{D})^{s} + N(j- i\sqrt{D}) \Mod{p}\label{e5}\\
	jm+knD &\equiv j(m+n\sqrt{D})^{s} -N\sqrt{D}(j-i\sqrt{D})\Mod{p} \label{e6}\\
	in+lm &\equiv l(m+n\sqrt{D})^{s} + N(k - l\sqrt{D}) \Mod{p} \label{e7}\\
	jn+km&\equiv  k(m+n\sqrt{D})^{s} - N\sqrt{D}(k-l\sqrt{D})\Mod{p} \label{e8}.	
	\end{align}
	Now, applying $(\ref{e6})+ \sqrt{D}(\ref{e5})$, and $(\ref{e8})+ \sqrt{D}(\ref{e7})$ we get,
	\begin{align}
	m(j+i\sqrt{D}) + nD(k+l\sqrt{D}) &\equiv (j+i\sqrt{D})(m+n\sqrt{D})^{s} \Mod{p} \label{e9}\\
	n(j+i\sqrt{D}) + m(k+l\sqrt{D}) &\equiv (k+l\sqrt{D})(m+n\sqrt{D})^{s} \Mod{p} \label{e10}.
	\end{align}
	Applying $(j+i\sqrt{D})\times (\ref{e10}) - (k+l\sqrt{D})\times (\ref{e9})$, we get
	\begin{align}
	0 &\equiv n((j+i\sqrt{D})^{2} - D(k+l\sqrt{D})^{2})(m+n\sqrt{D})^{s} \Mod{p}\\
	&\equiv (j^{2}+ Di^{2}- Dk^{2} - D^{2}l^{2}) + \sqrt{D}(2ij- 2klD)\Mod{p}.
	\end{align}
	This implies that $j^{2}+ Di^{2}- Dk^{2} - D^{2}l^{2} \equiv 0\Mod{p}$ and $2ij- 2klD \equiv 0\Mod{p}$. So, letting $k = ri$, $j = rlD$ for some $r\in GF(p)$, we get $r^{2}l^{2}D^{2} + Di^{2} - Dr^{2}i^{2} - D^{2}l^{2}\equiv 0\Mod{p}$. Then, $(r^{2}-1)(l^{2}D^{2}-Di^{2})\equiv 0\Mod{p}$. Note that, $l^{2}D^{2}-Di^{2} = -D(i+l\sqrt{D})(i-l\sqrt{D})$. So, if $l^{2}D^{2}-Di^{2} \equiv 0 \Mod{p}$, then either $i+l\sqrt{D} \equiv 0\Mod{p}$ or  $i-l\sqrt{D} \equiv 0\Mod{p}$, but this is not possible. So, $r^{2}\equiv 1\Mod{p}$ which implies $r\equiv \pm 1\Mod{p}$. Therefore, $k = \pm i$, $j = \pm lD$. 
	
	Now, applying $m\times (\ref{e10}) - n\times (\ref{e9})$, we get $$(k+l\sqrt{D})(m^{2}- Dn^{2}) \equiv (m+n\sqrt{D})^{s}(m(k+l\sqrt{D}) - n(j+i\sqrt{D})\Mod{p}.$$ Then, $(\pm i + l\sqrt{D})(m^{2}- Dn^{2}) \equiv (m+n\sqrt{D})^{s}(m(\pm i+l\sqrt{D}) - n(\pm lD+i\sqrt{D}))\Mod{p}$ which implies $m \pm n\sqrt{D} \equiv (m+n\sqrt{D})^{s} \Mod{p}$. Since, $m \pm n\sqrt{D} \in \langle m+n\sqrt{D}\rangle$, $m^{2}-Dn^{2}\in \langle m+n\sqrt{D}\rangle$. But, one can easily observe that $1\in \langle m+n\sqrt{D}\rangle$ is the only element without involving $\sqrt{D}$, as $n\ne 0$. So, $m^{2}-Dn^{2} = 1$ and $(m+n\sqrt{D})^{-1} = m-n\sqrt{D}$. Thus, $(m+n\sqrt{D})^{s} \equiv (m+n\sqrt{D})^{\pm 1} \Mod{p}$ and so, $s \equiv \pm 1\Mod{q}$. Hence, $\alpha=  \begin{pmatrix}
	i& l\\ \pm lD & \pm i
	\end{pmatrix}$ and $\delta(a) = a^{s}$, where $1\le l,j \le p-1$ and $s \equiv \pm 1\Mod{q}$ and so,  $R \simeq \langle \left(\begin{pmatrix}
	i & l\\ ld & i
	\end{pmatrix}, \delta_{s}\right), \left(\begin{pmatrix}
	i & l\\ -ld & -i
	\end{pmatrix}, \delta_{-s}\right)\rangle \simeq (\mathbb{Z}_{p^{2}-1}\times \mathbb{Z}_{q})\rtimes \mathbb{Z}_{2}$, where $\delta_{s}(a) = a^{s}$ and $s\equiv \pm1 \Mod{p}$. 
	
	Now. let $\beta\in S$ be defined by $\beta(a) = b^{\lambda}c^{\rho}$, where $0\le \lambda, \rho \le p-1$. Then for any $0\le l\le q^{2}-1$, we get
	\begin{align*}
	\beta(a^{l}) = b^{\lambda X_{l} + \rho Y_{l}}c^{\rho X_{l} + \lambda DY_{l}},
	\end{align*}
	where $$X_{l} = \sum_{u=0}^{[\frac{l-1}{2}]}{n^{2u}D^{u}\sum_{v=2u}^{l-1}{^{v}C_{2u}{m^{v-2u}}}}\; \text{and}\; Y_{l} = \sum_{u=1}^{[\frac{l-1}{2}]}{n^{2u-1}D^{u-1}\sum_{v=2u-1}^{l-1}{^{v}C_{2u-1}{m^{v-2u+1}}}}.$$
	Now, $1 = \beta(a^{q^{2}}) = b^{\lambda X_{q^{2}} + \rho Y_{q^{2}}}c^{\rho X_{q^{2}} + \lambda DY_{q^{2}}}$. Then $\lambda X_{q^{2}} + \rho Y_{q^{2}} \equiv 0\Mod{p}$ and $\rho X_{q^{2}} + \lambda DY_{q^{2}} \equiv 0\Mod{p}$. Therefore, $(\rho^{2}-D\lambda^{2})(Y_{q^{2}})\equiv 0\Mod{p}$ which implies that $Y_{q^{2}}\equiv 0 \Mod{p}$. Similarly, $X_{q^{2}} \equiv 0\Mod{p}$ and so, $\beta(a^{q^{2}}) = 1$. Thus $Q\simeq \mathbb{Z}_{p}\times \mathbb{Z}_{p}$. Hence, using the Theorem \ref{main}, $Aut(G)\simeq (\mathbb{Z}_{p}\times \mathbb{Z}_{p})\rtimes ((\mathbb{Z}_{p^{2}-1}\times \mathbb{Z}_{q})\rtimes \mathbb{Z}_{2})$. \label{g30}
	
		\subsection{Type 31.}  Let $H = \langle b, c\mid b^{p} = c^{p}=1, bc=cb\rangle \simeq \mathbb{Z}_{p}\times \mathbb{Z}_{p}$, $K = \langle a \mid a^{4} =1\rangle \simeq \mathbb{Z}_{4}$ and the homomorphism $\phi: K \longrightarrow Aut(H)$ be defined as $\phi(a)(b) = a^{-1}ba = b^{m}c^{nD}, \phi(a)(c) = a^{-1}ca = b^{n}c^{m}$, where $m + n\sqrt{D} = \sigma^{\frac{p^{2}-1}{4}}$, $\sigma$ is a primitive root of Galois Field $GF(p^{2})$, $m, n, D\in GF(p)$, $n \ne 0$, $D$ is not a perfect square and $p \equiv 3\Mod{4}$. Then $G\simeq H\rtimes_{\phi} K$. Note that, $(m+n\sqrt{D})^{2} \equiv -1\Mod{p}$ which implies that $m^{2}+n^{2}D+ 2mn\sqrt{D}\equiv -1 \Mod{p}$. Therefore, $m^{2}+n^{2}D = -1$ and $2mn = 0$. Since $n\ne 0$, $m = 0$ and $n^{2}D = -1$. Thus, $b^{a} = c^{nD}$ and $c^{a} = b^{n}$. Now, let $(\alpha, \delta)\in P$. Then $\delta\in Aut(K)$ is defined by $\delta(a) = a^{s}$ and $\alpha\in Aut(H) \simeq GL(2,p)$ is given by $\alpha = \begin{pmatrix}
	i& l\\ j & k
	\end{pmatrix}$ such that $\alpha(b) = b^{i}c^{j}, \alpha(c) = b^{l}c^{k}$, where $s\in \{1,3\}$, $0\le i,j,l,k \le p-1$ and $ik\not\equiv jl \Mod{p}$. Since, $(\alpha, \delta)\in P$, $\alpha(b^{a}) = \alpha(b)^{\delta(a)}$ and $\alpha(c^{a}) = \alpha(c)^{\delta(a)}$.
	
	Now, $\alpha(b^{a}) =  \alpha(c^{nD}) = b^{lnD}c^{knD}$ and $\alpha(c^{a}) = \alpha(b^{n}) = b^{in}c^{jn}$, $\alpha(b)^{\delta(a)} = (b^{i}c^{j})^{a^s} = \left\{\begin{array}{ll}
	b^{jn}c^{inD}& \text{if} \; s = 1\\
	b^{-jn}c^{-inD}& \text{if} \; s = 3
	\end{array} \right.$ and $\alpha(c)^{\delta(a)} = (b^{l}c^{k})^{a^s} = \left\{\begin{array}{ll}
	b^{kn}c^{lnD}& \text{if} \; s = 1\\
	b^{-kn}c^{-lnD}& \text{if} \; s = 3
	\end{array} \right.$. 
	
	If $s= 1$, then $j = lD$ and $k = i$, and if $s=3$, then $j= -lD$ and $k= -i$. Thus, by the similar argument in above part, we get $R \simeq \mathbb{Z}_{p^{2}-1}\rtimes \mathbb{Z}_{2}$ and $Q\simeq \mathbb{Z}_{p}\times \mathbb{Z}_{p}$. Hence, using the Theorem \ref{main}, $Aut(G)\simeq (\mathbb{Z}_{p}\times \mathbb{Z}_{p})\rtimes (\mathbb{Z}_{p^{2}-1}\rtimes \mathbb{Z}_{2})$. \label{g31}
	
		\subsection{Type 32.} Let $H = \langle b, c\mid b^{p} = c^{p}=1, bc=cb\rangle \simeq \mathbb{Z}_{p}\times \mathbb{Z}_{p}$, $K = \langle a \mid a^{q^2} =1\rangle \simeq \mathbb{Z}_{q^{2}}$ and the homomorphism $\phi: K \longrightarrow Aut(H)$ be defined as $\phi(a)(b) = a^{-1}ba = b^{m}c^{nD}, \phi(a)(c) = a^{-1}ca = b^{n}c^{m}$, where $m + n\sqrt{D} = \sigma^{\frac{p^{2}-1}{q^{2}}}$, $\sigma$ is a primitive root of Galois field $GF(p^{2})$, $m, n, D\in GF(p)$, $n \ne 0$, $D$ is not a perfect square and $p \equiv 3\Mod{4}$. Then $G\simeq H\rtimes_{\phi} K$.
Using the similar argument as in \ref{g31}, we get $R \simeq \mathbb{Z}_{p^{2}-1}\rtimes \mathbb{Z}_{2}$ and $Q\simeq \mathbb{Z}_{p}\times \mathbb{Z}_{p}$. Hence, using the Theorem \ref{main}, $Aut(G)\simeq (\mathbb{Z}_{p}\times \mathbb{Z}_{p})\rtimes (\mathbb{Z}_{p^{2}-1}\rtimes \mathbb{Z}_{2})$. \label{g32}

	\subsection{Type 33.}  Let $H = \langle c, d\mid c^{p} = d^{p}=1, cd=dc\rangle \simeq \mathbb{Z}_{p}\times \mathbb{Z}_{p}$, $K = \langle a, b\mid a^{q} = b^{q}=1, ab=ba\rangle \simeq \mathbb{Z}_{q}\times \mathbb{Z}_{q}$ and the homomorphism $\phi: K \longrightarrow Aut(H)$ be defined as $\phi(a)(c) = a^{-1}ca = c, \phi(a)(d) = a^{-1}da = d, \phi(b)(c) = b^{-1}cb = c^{r}, \phi(b)(d) = b^{-1}db = d^{r}$, where $r^{q}\equiv 1 \Mod{p}$. Then $G\simeq H\rtimes_{\phi} K$. Now, let $(\alpha, \delta)\in P$. Then $\alpha\in Aut(H) \simeq GL(2,p)$ is given by $\alpha = \begin{pmatrix}
i& l\\ j & k
\end{pmatrix}$ and $\delta\in Aut(K) \simeq GL(2,q)$ is given by $\delta = \begin{pmatrix}
m& s\\ n & t
\end{pmatrix}$ such that $\alpha(c) = c^{i}d^{j}, \alpha(d) = c^{l}d^{k}, \delta(a) = a^{m}b^{n}, \delta(b) = a^{s}b^{t}$,  where $0\le i,j,l,k \le p-1$ and $ik\not\equiv jl \Mod{p}$, and $0\le m,n,s,t \le q-1$ and $mt\not\equiv ns \Mod{q}$. Since, $(\alpha, \delta)\in P$, $\alpha(c^{a}) = \alpha(c)^{\delta(a)}$, $\alpha(d^{a}) = \alpha(d)^{\delta(a)}$, $\alpha(c^{b}) = \alpha(c)^{\delta(b)}$ and $\alpha(d^{b}) = \alpha(d)^{\delta(b)}$.

Now, $c^{i}d^{j} = \alpha(c) = \alpha(c^{a}) = \alpha(c)^{\delta(a)} = (c^{i}d^{j})^{a^{m}b^{n}} = c^{ir^{n}}d^{jr^{n}}$, $c^{l}d^{k} = \alpha(d) = \alpha(d^{a}) = \alpha(d)^{\delta(a)} = (c^{l}d^{k})^{a^{m}b^{n}} = c^{lr^{n}}d^{kr^{n}}$, $c^{ir}d^{jr} = \alpha(c^{r}) = \alpha(c^{b}) = \alpha(c)^{\delta(b)} = (c^{i}d^{j})^{a^{s}b^{t}} = c^{ir^{t}}d^{jr^{t}}$ and $c^{lr}d^{kr} = \alpha(d^{r}) = \alpha(d^{b}) = \alpha(d)^{\delta(b)} = (c^{l}d^{k})^{a^{s}b^{t}} = c^{lr^{t}}d^{kr^{t}}$. Then $i\equiv ir^{n}\Mod{p}$, $j\equiv jr^{n}\Mod{p}$, $l\equiv lr^{n}\Mod{p}$, $k\equiv kr^{n}\Mod{p}$, $ir\equiv ir^{t}\Mod{p}$, $jr\equiv jr^{t}\Mod{p}$, $lr\equiv lr^{t}\Mod{p}$ and $kr\equiv kr^{t}\Mod{p}$. Since all $i,j,k,l$ can not vanish together, let $i\ne 0$. Then $r^{n}\equiv 1\Mod{p}$ and $r^{t}\equiv r\Mod{p}$ which implies that $n\equiv 0\Mod{q}$ and $t\equiv 1 \Mod{q}$. Thus, $\alpha = \begin{pmatrix}
i& l\\ j & k
\end{pmatrix}$ and $\delta = \begin{pmatrix}
m& 0\\ s & 1
\end{pmatrix}$, where $0\le i,j,k,l \le p-1$, $ik\not\equiv jl \Mod{p}$, $1\le m\le q-1$ and $0\le s\le q-1$. Therefore, $R\simeq GL(2,p)\times (\mathbb{Z}_{q}\rtimes \mathbb{Z}_{q-1})$.

Now, let $\beta\in S$ be defined by $\beta(a) = c^{\lambda}d^{\mu}$ and $\beta(b) = c^{\rho}d^{\nu}$, where $0\le \lambda, \mu, \rho, \nu\le p-1$. Now for any $0\le u\le q-1$, we have $\beta(a^{u}) = c^{u\lambda}d^{u\mu}$ and $\beta(b^{u}) = c^{\lambda \frac{r^{u}-1}{r-1}}d^{\mu\frac{r^{u}-1}{r-1}}$. Then clearly, $\beta(a^{q}) = c^{q\lambda}d^{q\mu}$ and $\beta(b^{q}) = c^{\lambda \frac{r^{q}-1}{r-1}}d^{\mu\frac{r^{q}-1}{r-1}} = 1$. Since $\beta(a^{q}) = 1$, $q\lambda, q\mu \equiv 0 \Mod{p}$ which implies that $\lambda, \mu \equiv 0\Mod{p}$. Thus, $\beta(a) = 1$ and $\beta(b) = c^{\rho}d^{\nu}$, where $0\le \rho, \nu\le p-1$ and so, $Q\simeq \mathbb{Z}_{p}\times \mathbb{Z}_{p}$. Hence, using the Theorem \ref{main}, $Aut(G)\simeq (\mathbb{Z}_{p}\times \mathbb{Z}_{p})\rtimes (GL(2,p)\times (\mathbb{Z}_{q}\rtimes \mathbb{Z}_{q-1}))$.  \label{g33}

	\subsection{Type 34.}  Let $H = \langle c, d\mid c^{p} = d^{p}=1, cd=dc\rangle \simeq \mathbb{Z}_{p}\times \mathbb{Z}_{p}$, $K = \langle a, b\mid a^{2} = b^{2}=1, ab=ba\rangle \simeq \mathbb{Z}_{2}\times \mathbb{Z}_{2}$ and the homomorphism $\phi: K \longrightarrow Aut(H)$ be defined as $\phi(a)(c) =  a^{-1}ca = c, \phi(a)(d) = a^{-1}da = d, \phi(b)(c) = b^{-1}cb = c^{-1}, \phi(b)(d) = b^{-1}db = d$. Then $G\simeq H\rtimes_{\phi} K$. Now, let $(\alpha, \delta)\in P$. Then $\alpha\in Aut(H) \simeq GL(2,p)$ is given by $\alpha = \begin{pmatrix}
i& l\\ j & k
\end{pmatrix}$ and $\delta\in Aut(K) \simeq GL(2,q)$ is given by $\delta = \begin{pmatrix}
m& s\\ n & t
\end{pmatrix}$ such that $\alpha(c) = c^{i}d^{j}, \alpha(d) = c^{l}d^{k}, \delta(a) = a^{m}b^{n}, \delta(b) = a^{s}b^{t}$,  where $0\le i,j,k,l \le p-1$, $ik\not\equiv jl \Mod{p}$, $0\le m,n,s,t \le 1$ and $mt\not\equiv ns \Mod{2}$. Since, $(\alpha, \delta)\in P$, $\alpha(c^{a}) = \alpha(c)^{\delta(a)}$, $\alpha(d^{a}) = \alpha(d)^{\delta(a)}$, $\alpha(c^{b}) = \alpha(c)^{\delta(b)}$ and $\alpha(d^{b}) = \alpha(d)^{\delta(b)}$.

Now, $c^{i}d^{j} = \alpha(c) = \alpha(c^{a}) = \alpha(c)^{\delta(a)} = (c^{i}d^{j})^{a^{m}b^{n}} = c^{i(-1)^{n}}d^{j}$, $c^{l}d^{k} = \alpha(d) = \alpha(d^{a}) = \alpha(d)^{\delta(a)} = (c^{l}d^{k})^{a^{m}b^{n}} = c^{l(-1)^{n}}d^{k}$, $c^{-i}d^{-j} = \alpha(c^{-1}) = \alpha(c^{b}) = \alpha(c)^{\delta(b)} = (c^{i}d^{j})^{a^{s}b^{t}} = c^{i(-1)^{t}}d^{j}$ and $c^{l}d^{k} = \alpha(d) = \alpha(d^{b}) = \alpha(d)^{\delta(b)} = (c^{l}d^{k})^{a^{s}b^{t}} = c^{l(-1)^{t}}d^{k}$. Then $i\equiv i(-1)^{n}\Mod{p}$, $l\equiv l(-1)^{n}\Mod{p}$, $-i\equiv i(-1)^{t}\Mod{p}$, $-j\equiv j\Mod{p}$, $l\equiv l(-1)^{t}\Mod{p}$. Clearly, $j=0$. So, $i\ne 0$. Therefore, $(-1)^{n}\equiv 1 \Mod{p}$ and $(-1)^{t}\equiv -1\Mod{p}$ which implies that $ n =0 $ and $t= 1$ and so, $l= 0$ and $m = 1$. Thus, $\alpha = \begin{pmatrix}
i& 0\\ 0 & k
\end{pmatrix}$ and $\delta = \begin{pmatrix}
1& s\\ 0 & 1
\end{pmatrix}$, where $1\le i,k \le p-1$ and $s\in\{0,1\}$. Therefore, $R\simeq (\mathbb{Z}_{p-1}\times \mathbb{Z}_{p-1})\times \mathbb{Z}_{2}$.

Now, let $\beta\in S$ be defined by $\beta(a) = c^{\lambda}d^{\mu}$ and $\beta(b) = c^{\rho}d^{\nu}$, where $0\le \lambda, \mu, \rho, \nu\le p-1$. Then $\beta(a^{2}) = c^{2\lambda}d^{2\mu}$. Since, $\beta(a^{2}) = 1$, $\lambda, \mu \equiv 0\Mod{p}$. Also, $\beta(b^{2}) = d^{2\nu}$, and $\beta(b^{2}) = 1$. So, $\nu \equiv 0\Mod{p}$. Thus, $\beta(a) = 1$ and $\beta(b) = c^{\rho}$, where $0\le \rho \le p-1$ and so, $Q\simeq \mathbb{Z}_{p}$. Hence, using the Theorem \ref{main}, $Aut(G)\simeq \mathbb{Z}_{p}\rtimes ((\mathbb{Z}_{p-1}\times \mathbb{Z}_{p-1})\times \mathbb{Z}_{2})$.  \label{g34}

	\subsection{Type 35.}   Let $H = \langle c, d\mid c^{p} = d^{p}=1, cd=dc\rangle \simeq \mathbb{Z}_{p}\times \mathbb{Z}_{p}$, $K = \langle a, b\mid a^{q} = b^{q}=1, ab=ba\rangle \simeq \mathbb{Z}_{q}\times \mathbb{Z}_{q}$ and the homomorphism $\phi: K \longrightarrow Aut(H)$ be defined as $\phi(a)(c) = a^{-1}ca = c^{r}, \phi(a)(d) = a^{-1}da = d, \phi(b)(c) = b^{-1}cb = c, \phi(b)(d) = b^{-1}db = d^{r}$, where $r^{q}\equiv 1 \Mod{p}$. Then $G\simeq H\rtimes_{\phi} K$. Now, let $(\alpha, \delta)\in P$. Then $\alpha\in Aut(H) \simeq GL(2,p)$ is given by $\alpha = \begin{pmatrix}
i& l\\ j & k
\end{pmatrix}$ and $\delta\in Aut(K) \simeq GL(2,q)$ is given by $\delta = \begin{pmatrix}
m& s\\ n & t
\end{pmatrix}$ such that $\alpha(c) = c^{i}d^{j}, \alpha(d) = c^{l}d^{k}, \delta(a) = a^{m}b^{n}, \delta(b) = a^{s}b^{t}$,  where $0\le i,j,k,l \le p-1$, $ik\not\equiv jl \Mod{p}$, $0\le m,n,s,t \le 1$ and $mt\not\equiv ns \Mod{2}$. Since, $(\alpha, \delta)\in P$, $\alpha(c^{a}) = \alpha(c)^{\delta(a)}$, $\alpha(d^{a}) = \alpha(d)^{\delta(a)}$, $\alpha(c^{b}) = \alpha(c)^{\delta(b)}$ and $\alpha(d^{b}) = \alpha(d)^{\delta(b)}$.

Now, $c^{ri}d^{rj} = \alpha(c^{r}) = \alpha(c^{a}) = \alpha(c)^{\delta(a)} = (c^{i}d^{j})^{a^{m}b^{n}} = c^{ir^{m}}d^{jr^{n}}$, $c^{l}d^{k} = \alpha(d) = \alpha(d^{a}) = \alpha(d)^{\delta(a)} = (c^{l}d^{k})^{a^{m}b^{n}} = c^{lr^{m}}d^{kr^{n}}$, $c^{i}d^{j} = \alpha(c) = \alpha(c^{b}) = \alpha(c)^{\delta(b)} = (c^{i}d^{j})^{a^{s}b^{t}} = c^{ir^{s}}d^{jr^{t}}$ and $c^{lr}d^{kr} = \alpha(d^{r}) = \alpha(d^{b}) = \alpha(d)^{\delta(b)} = (c^{l}d^{k})^{a^{s}b^{t}} = c^{lr^{s}}d^{kr^{t}}$. Then $ri\equiv ir^{m}\Mod{p}$, $rj\equiv jr^{n}\Mod{p}$, $l\equiv lr^{m}\Mod{p}$, $k\equiv kr^{n}\Mod{p}$, $i\equiv ir^{s}\Mod{p}$, $j\equiv jr^{t}\Mod{p}$, $lr\equiv lr^{s}\Mod{p}$ and $kr\equiv kr^{t}\Mod{p}$. Since $i$ and $j$ both can not vanish together, we have two cases namely, $i\ne 0$ and $j\ne 0$. Also, we will observe that one of $i$ or $j$ must be zero.

If $i\ne 0$, then $m = 1$, $s = 0$, $l = 0$. Since, $l=0$, $k\ne 0$ and so, $n = 0$ and $t = 1$ which implies that $j=0$. Thus in this case, $\alpha = \begin{pmatrix}
i& 0\\ 0 & k
\end{pmatrix}$ and $\delta = \begin{pmatrix}
1& 0\\ 0 & 1
\end{pmatrix}$.  Now, if $j\ne 0$, then $n=1$, $t=0$ and so, $k=0$. Since, $k=0$, $l\ne 0$ and so, $s=1$, $m=0$ which implies that $i=0$. Thus in this case, $\alpha = \begin{pmatrix}
0& l\\ j & 0
\end{pmatrix}$ and $\delta = \begin{pmatrix}
0& 1\\ 1 & 0
\end{pmatrix}$. Therefore, $R\simeq \langle \left(\begin{pmatrix}
i& 0\\ 0 & k
\end{pmatrix}, \begin{pmatrix}
	1& 0\\ 0 & 1
\end{pmatrix}\right), \left(\begin{pmatrix}
0& l\\ j & 0
\end{pmatrix}, \begin{pmatrix}
0& 1\\ 1& 0
\end{pmatrix}\right)\rangle \simeq (\mathbb{Z}_{p-1}\times \mathbb{Z}_{p-1})\rtimes \mathbb{Z}_{2}$.

Now, let $\beta\in S$ be defined by $\beta(a) = c^{\lambda}d^{\mu}$ and $\beta(b) = c^{\rho}d^{\nu}$, where $0\le \lambda, \mu, \rho, \nu\le p-1$. Now for any $0\le u\le q-1$, we have $\beta(a^{u}) = c^{\lambda\frac{r^{u}-1}{r-1}}d^{u\mu}$ and $\beta(b^{u}) = c^{u\rho}d^{\nu\frac{r^{u}-1}{r-1}}$. Then clearly, $\beta(a^{q}) = d^{q\mu}$ and $\beta(b^{q}) = c^{q\rho}$. Since $\beta(a^{q}) = 1 = \beta(b^{q})$, $q\rho, q\mu \equiv 0 \Mod{p}$ which implies that $\rho, \mu \equiv 0\Mod{p}$. Thus, $\beta(a) = c^{\lambda}$ and $\beta(b) = d^{\nu}$, where $0\le \lambda, \nu\le p-1$ and so, $Q\simeq \mathbb{Z}_{p}\times \mathbb{Z}_{p}$. Hence, using the Theorem \ref{main}, $Aut(G)\simeq (\mathbb{Z}_{p}\times \mathbb{Z}_{p})\rtimes ((\mathbb{Z}_{p-1}\times \mathbb{Z}_{p-1})\rtimes \mathbb{Z}_{2})$. \label{g35}

	\subsection{Type 36.}  Let $H = \langle c, d\mid c^{p} = d^{p}=1, cd=dc\rangle \simeq \mathbb{Z}_{p}\times \mathbb{Z}_{p}$, $K = \langle a, b\mid a^{q} = b^{q}=1, ab=ba\rangle \simeq \mathbb{Z}_{q}\times \mathbb{Z}_{q}$ and the homomorphism $\phi: K \longrightarrow Aut(H)$ be defined as $\phi(a)(c) =  a^{-1}ca = c, \phi(a)(d) = a^{-1}da = d, \phi(b)(c) = b^{-1}cb = c^{u}d^{vD}, \phi(b)(d) = b^{-1}db = c^{v}d^{u}$, where $u + v\sqrt{D} = \sigma^{\frac{p^{2}-1}{q}}$, $\sigma$ is a primitive root of Galois Field $GF(p^{2})$, $u, v, D\in GF(p)$, $v \ne 0$, $D$ is not a perfect square, $q$ divides $p+1$ and $p\not\equiv 1 \Mod{q}$. Then $G\simeq H\rtimes_{\phi} K$. Now, let $(\alpha, \delta)\in P$. Then $\alpha\in Aut(H) \simeq GL(2,p)$ is given by $\alpha = \begin{pmatrix}
i& l\\ j & k
\end{pmatrix}$ and $\delta\in Aut(K) \simeq GL(2,q)$ is given by $\delta = \begin{pmatrix}
m& s\\ n & t
\end{pmatrix}$ such that $\alpha(c) = c^{i}d^{j}, \alpha(d) = c^{l}d^{k}, \delta(a) = a^{m}b^{n}, \delta(b) = a^{s}b^{t}$,  where $0\le i,j,k,l \le p-1$, $ik\not\equiv jl \Mod{p}$, $0\le m,n,s,t \le 1$ and $mt\not\equiv ns \Mod{2}$. Since, $(\alpha, \delta)\in P$, $\alpha(c^{a}) = \alpha(c)^{\delta(a)}$, $\alpha(d^{a}) = \alpha(d)^{\delta(a)}$, $\alpha(c^{b}) = \alpha(c)^{\delta(b)}$ and $\alpha(d^{b}) = \alpha(d)^{\delta(b)}$.

Now, $c^{i}d^{j} = \alpha(c) = \alpha(c^{a}) = \alpha(c)^{\delta(a)} = (c^{i}d^{j})^{a^{m}b^{n}} = c^{iM+ jN}d^{iND+ jM}$, $ c^{l}d^{k} = \alpha(d) = \alpha(d^{a}) = \alpha(d)^{\delta(a)} = (c^{l}d^{k})^{a^{m}b^{n}} = c^{lM+ kN}d^{lND+ kM}$, $ c^{iu+lvD}d^{ju+kvD} = \alpha(c^{u}d^{vD}) = \alpha(c^{b})  = \alpha(c)^{\delta(b)} = (c^{i}d^{j})^{a^{s}b^{t}} = c^{iX+ jY}d^{iYD+ jX}$, and $c^{iv+lu}d^{jv+ku} = \alpha(c^{v}d^{u}) = \alpha(d^{b}) = \alpha(d)^{\delta(b)} = (c^{l}d^{k})^{a^{s}b^{t}} = c^{lX+ kY}d^{lYD+ kX}$, where $$M = \sum_{x=0}^{[\frac{n}{2}]}{{^{n}}C_{2x}u^{n-2x}v^{2x}D^{x}}, N = \sum_{x=0}^{[\frac{n}{2}]}{^{n}C_{2x+1}v^{n-2x-1}v^{2x+1}D^{x}},$$ $$X = \sum_{e=0}^{[\frac{t}{2}]}{{^{t}}C_{2e}u^{t-2e}v^{2e}D^{e}},\; \text{and}\; Y = \sum_{e=0}^{[\frac{t}{2}]}{^{t}C_{2e+1}m^{t-2e-1}n^{2e+1}D^{e}}.$$ Then 
	\begin{align}
	i\equiv iM+jN\Mod{p}\label{f1}\\
		j\equiv iND+jM\Mod{p}\label{f2}\\
			l\equiv lM+kN\Mod{p}\label{f3}\\
				k\equiv lND+kM\Mod{p}\label{f4}\\
iu+lvD \equiv iX + jY \Mod{p}\label{f5}\\
ju+kvD \equiv iYD+ jX\Mod{p} \label{f6}\\
iv+lu \equiv lX+ kY \Mod{p} \label{f7}\\
jv+ku\equiv lYD + kX \Mod{p} \label{f8}.	
\end{align}
Note that $(u+v\sqrt{D})^{n} = M + \sqrt{D}N$ and $(u+v\sqrt{D})^{t} = X + \sqrt{D}Y$. Now, substituting the value of $M$ and $X$ in the Congruence relations (\ref{f1})-(\ref{f8}), we get   
\begin{align}
i\equiv i(u+v\sqrt{D})^{n}+N(j-i\sqrt{D})\Mod{p}\label{f9}\\
j\equiv j(u+v\sqrt{D})^{n}-N\sqrt{D}(j-i\sqrt{D})\Mod{p}\label{f10}\\
l\equiv  l(u+v\sqrt{D})^{n}+N(k-l\sqrt{D})\Mod{p}\label{f11}\\
k\equiv  k(u+v\sqrt{D})^{n}-N\sqrt{D}(k-l\sqrt{D})\Mod{p}\label{f12}\\
iu+lvD \equiv i(u+v\sqrt{D})^{t}+Y(j-i\sqrt{D}) \Mod{p}\label{f13}\\
ju+kvD \equiv j(u+v\sqrt{D})^{t}-Y\sqrt{D}(j-i\sqrt{D})\Mod{p} \label{f14}\\
iv+lu \equiv l(u+v\sqrt{D})^{t}+Y(k-l\sqrt{D}) \Mod{p} \label{f15}\\
jv+ku\equiv k(u+v\sqrt{D})^{t}-Y\sqrt{D}(k-l\sqrt{D}) \Mod{p} \label{f16}.	
\end{align}
Solving the congruence relations (\ref{f9} - \ref{f16}) we get,
\begin{align}
j+i\sqrt{D} \equiv (j+i\sqrt{D})(u+v\sqrt{D})^{n}\Mod{p}\label{g1}\\
k+l\sqrt{D} \equiv (k+l\sqrt{D})(u+v\sqrt{D})^{n}\Mod{p}\label{g2}\\
u(j+i\sqrt{D}) + vD(k+l\sqrt{D}) \equiv (j+i\sqrt{D})(u+v\sqrt{D})^{t} \Mod{p} \label{g3}\\
v(j+i\sqrt{D}) + u(k+l\sqrt{D}) \equiv (k+l\sqrt{D})(u+v\sqrt{D})^{t} \Mod{p} \label{g4}.
\end{align}
From the congruence relations (\ref{g1}) and (\ref{g2}) we get, $(u+v\sqrt{D})^{n}\equiv 1 \Mod{p}$ which implies that $n \equiv 0\Mod{q}$. 

Applying $(j+i\sqrt{D})\times (\ref{g4}) - (k+l\sqrt{D})\times (\ref{g3})$, we get
\begin{align}
0 \equiv& v((j+i\sqrt{D})^{2} - D(k+l\sqrt{D})^{2})(u+v\sqrt{D})^{t} \Mod{p}\\
\equiv& (j^{2}+ Di^{2}- Dk^{2} - D^{2}l^{2}) + \sqrt{D}(2ij- 2lkD)\Mod{p}.
\end{align}
		This implies that $j^{2}+ Di^{2}- Dk^{2} - D^{2}l^{2} \equiv 0\Mod{p}$ and $2ij- 2lkD\equiv 0\Mod{p}$. So, letting $k = ri$, $j = rlD$ for some $r\in GF(p)$, we get $r^{2}l^{2}D^{2} + Di^{2} - Dr^{2}i^{2} - D^{2}l^{2}\equiv 0\Mod{p}$. Then, $(r^{2}-1)(l^{2}D^{2}-Di^{2})\equiv 0\Mod{p}$. Note that, $l^{2}D^{2}-Di^{2} = -D(i+l\sqrt{D})(i-l\sqrt{D})$. So, if $l^{2}D^{2}-Di^{2} \equiv 0 \Mod{p}$, then either $i+l\sqrt{D} \equiv 0\Mod{p}$ or  $i-l\sqrt{D} \equiv 0\Mod{p}$, but this is not possible. So, $r^{2}\equiv 1\Mod{p}$ which implies $r\equiv \pm 1\Mod{p}$. Therefore, $k = \pm i$, $j = \pm lD$. 
		
		Now, applying $u\times (\ref{g4}) - v\times (\ref{g3})$, we get $(k+l\sqrt{D})(u^{2}- Dv^{2}) \equiv (u+v\sqrt{D})^{t}(u(k+l\sqrt{D}) - v(j+i\sqrt{D}))\Mod{p}$. Then, $(\pm i + l\sqrt{D})(u^{2}- Dv^{2}) \equiv (u+v\sqrt{D})^{t}(u(\pm i+l\sqrt{D}) - v(\pm lD+i\sqrt{D}))\Mod{p}$ which implies $u \pm v\sqrt{D} \equiv (u+v\sqrt{D})^{t} \Mod{p}$. Since, $u \pm v\sqrt{D} \in \langle u+v\sqrt{D}\rangle$, $u^{2}-Dv^{2}\in \langle u+v\sqrt{D}\rangle$. But, one can easily observe that $1\in \langle u+v\sqrt{D}\rangle$ is the only element without involving $\sqrt{D}$, as $v\ne 0$. So, $u^{2}-Dv^{2} = 1$ and $(u+v\sqrt{D})^{-1} = u-v\sqrt{D}$. Thus, $(u+v\sqrt{D})^{t} \equiv (u+v\sqrt{D})^{\pm 1} \Mod{p}$ and so, $t \equiv \pm 1\Mod{q}$. Hence, $\alpha=  \begin{pmatrix}
		i& l\\ \pm lD & \pm i
		\end{pmatrix}$ and $\delta = \begin{pmatrix}
		m& s\\ 0 & \pm 1
		\end{pmatrix}$, where $1\le i,l\le p-1$, $1\le m\le q-1$ and $0\le s\le q-1$. Thus $R\simeq \mathbb{Z}_{2}\times (\mathbb{Z}_{p^{2}-1}\times (\mathbb{Z}_{q-1}\rtimes \mathbb{Z}_{q}))$.
		
		Now. let $\beta\in S$ be defined by $\beta(a) = c^{\lambda}d^{\mu}$, and $\beta(b) = c^{\rho}d^{\nu}$ where $0\le \lambda,\mu,\rho,\nu \le p-1$. Then for any $0\le w\le q-1$, we get
		\begin{align*}
		\beta(a^{w}) = c^{w\lambda}d^{w\mu}\; \text{and}\; \beta(b^{w}) = c^{\rho X_{w} + \nu Y_{w}}d^{\nu X_{w} + \rho DY_{w}},
		\end{align*}
		where $$X_{w} = \sum_{\theta=0}^{[\frac{w-1}{2}]}{v^{2\theta}D^{\theta}\sum_{z=2\theta}^{w-1}{^{z}C_{2\theta}{u^{z-2\theta}}}}\; \text{and}\; Y_{w} = \sum_{\theta=1}^{[\frac{w-1}{2}]}{v^{2\theta-1}D^{\theta-1}\sum_{z=2\theta-1}^{w-1}{^{z}C_{2\theta-1}{u^{z-2\theta+1}}}}.$$
		Now, $1 = \beta(a^{q}) = c^{q\lambda}d^{q\mu}$. Then $q\lambda, q\mu\equiv 0\Mod{p}$ which implies that $\lambda, \mu \equiv 0\Mod{p}$. Also, $1 = \beta(b^{q}) = c^{\rho X_{q} + \nu Y_{q}}d^{\nu X_{q} + \rho DY_{q}}$. Then $\rho X_{q} + \nu Y_{q} \equiv 0\Mod{p}$ and $\nu X_{q} + \rho DY_{q} \equiv 0\Mod{p}$ which implies that $(\nu^{2}-D\rho^{2})Y_{q}\equiv 0\Mod{p}$ and so, $Y_{q}\equiv 0 \Mod{p}$. Similarly, $X_{q}\equiv 0 \Mod{p}$. Thus $\beta(a) = 1$ and $\beta(b) = c^{\rho}d^{\nu}$ where $0\le \rho,\nu \le p-1$. Therefore, $Q\simeq \mathbb{Z}_{p}\times \mathbb{Z}_{p}$. Hence, using the Theorem \ref{main}, $Aut(G)\simeq (\mathbb{Z}_{p}\times \mathbb{Z}_{p})\rtimes ((\mathbb{Z}_{2}\times \mathbb{Z}_{p^{2}-1})\times (\mathbb{Z}_{q-1}\rtimes \mathbb{Z}_{q}))$. \label{g36}
		
\begin{sidewaystable}
	\centering
	\begin{tabular}{|c|c|c|c|}
		\hline
		Type & Conditions & Group Description $G$ & Structure of $Aut(G)$\\
		\hline
		\hline
		15 &  & $\mathbb{Z}_{p^{2}}\times \mathbb{Z}_{q^{2}}$ &  $\mathbb{Z}_{p(p-1)}\times \mathbb{Z}_{q(q-1)}$ \\
		\hline
		16 &  & $\mathbb{Z}_{q}\times \mathbb{Z}_{q}\times \mathbb{Z}_{p^{2}}$ & $GL(2,q)\times \mathbb{Z}_{p(p-1)}$  \\
		\hline
		17& & $\mathbb{Z}_{q^{2}} \times \mathbb{Z}_{p} \times \mathbb{Z}_{p}$ & $\mathbb{Z}_{q(q-1)} \times GL(2,p)$ \\
		\hline
		18& & $\mathbb{Z}_{q}\times \mathbb{Z}_{q}\times \mathbb{Z}_{p} \times \mathbb{Z}_{p}$ & $GL(2,q) \times GL(2,p)$  \\
		\hline
		19& $q|p-1$ & $\mathbb{Z}_{p^{2}}\rtimes \mathbb{Z}_{q^{2}}$ & $\mathbb{Z}_{p^{2}}\rtimes (\mathbb{Z}{p(p-1)}\times \mathbb{Z}_{q})$   \\
		\hline
		20&$q^{2}|p-1$ &  $\mathbb{Z}_{p^{2}}\rtimes \mathbb{Z}_{q^{2}}$ & $\mathbb{Z}_{p^{2}}\rtimes \mathbb{Z}{p(p-1)}$ \\
		\hline
		21&$q|p-1$ & $\mathbb{Z}_{p^{2}}\rtimes (\mathbb{Z}_{q}\times \mathbb{Z}_{q})$ & $\mathbb{Z}_{p^{2}}\rtimes (\mathbb{Z}_{p(p-1)}\times (\mathbb{Z}_{q}\rtimes \mathbb{Z}_{q-1}))$ \\
		\hline
		22& $q|p-1$ & $(\mathbb{Z}_{p}\times \mathbb{Z}_{p})\rtimes_{\phi_{r}} \mathbb{Z}_{q^{2}}$ & $(\mathbb{Z}_{p}\times \mathbb{Z}_{p})\rtimes (GL(2,p)\times \mathbb{Z}_{q})$ \\
		\hline
		23& $q=2$ & $(\mathbb{Z}_{p}\times \mathbb{Z}_{p})\rtimes \mathbb{Z}_{4}$ & $\mathbb{Z}_{p}\rtimes ((\mathbb{Z}_{p-1}\times \mathbb{Z}_{p-1})\times \mathbb{Z}_{2})$ \\
		\hline
		24& $q|p-1$, $q\ne 2$ & $\mathbb{Z}_{p}\rtimes (\mathbb{Z}_{p}\times \mathbb{Z}_{q^{2}}$ & $\mathbb{Z}_{p}\rtimes ((\mathbb{Z}_{p-1}\times \mathbb{Z}_{p-1})\times \mathbb{Z}_{q})$ \\
		\hline
		25& $q|p-1$ & $(\mathbb{Z}_{p}\times \mathbb{Z}_{p})\rtimes_{\phi_{r^{-1}}} \mathbb{Z}_{q^{2}}$ & $(\mathbb{Z}_{p}\times \mathbb{Z}_{p})\rtimes (((\mathbb{Z}_{p-1}\times \mathbb{Z}_{p-1})\times \mathbb{Z}_{q}) \rtimes \mathbb{Z}_{2})$ \\
		\hline
		26& $q^{2}|p-1$& $(\mathbb{Z}_{p}\times \mathbb{Z}_{p})\rtimes_{\phi_{r^{-1}}} \mathbb{Z}_{q^{2}}$ & $(\mathbb{Z}_{p}\times \mathbb{Z}_{p})\rtimes ((\mathbb{Z}_{p-1}\times \mathbb{Z}_{p-1}) \rtimes \mathbb{Z}_{2})$ \\
		\hline
		27& $q|p-1$ & $(\mathbb{Z}_{p}\times \mathbb{Z}_{p})\rtimes_{\phi_{r^{n}}} \mathbb{Z}_{q^{2}}$ & $(\mathbb{Z}_{p}\times \mathbb{Z}_{p})\rtimes ((\mathbb{Z}_{p-1}\times \mathbb{Z}_{p-1}) \rtimes \mathbb{Z}_{q})$ \\
		\hline
		28& $q^{2}|p-1$ & $(\mathbb{Z}_{p}\times \mathbb{Z}_{p})\rtimes_{\phi_{r^{n}}} \mathbb{Z}_{q^{2}}$ & $(\mathbb{Z}_{p}\times \mathbb{Z}_{p})\rtimes (\mathbb{Z}_{p-1}\times \mathbb{Z}_{p-1}) $ \\
		\hline
		29&  $q^{2}|p-1$ & $(\mathbb{Z}_{p}\times \mathbb{Z}_{p})\rtimes_{\phi} \mathbb{Z}_{q^{2}}$ & $(\mathbb{Z}_{p}\times \mathbb{Z}_{p})\rtimes GL(2,p)$  \\
		\hline
		30& $q|p+1$ & $(\mathbb{Z}_{p}\times \mathbb{Z}_{p})\rtimes \mathbb{Z}_{q^{2}}$ & $(\mathbb{Z}_{p}\times \mathbb{Z}_{p})\rtimes ((\mathbb{Z}_{p^{2}-1}\times \mathbb{Z}_{q})\rtimes \mathbb{Z}_{2}) $ \\
		\hline
		31& $p\equiv 3\Mod{4}$ & $(\mathbb{Z}_{p}\times \mathbb{Z}_{p})\rtimes \mathbb{Z}_{4}$ & $(\mathbb{Z}_{p}\times \mathbb{Z}_{p})\rtimes (\mathbb{Z}_{p^{2}-1}\rtimes \mathbb{Z}_{2}) $ \\
		\hline
		32& $q^{2}| p+1$ & $(\mathbb{Z}_{p}\times \mathbb{Z}_{p})\rtimes \mathbb{Z}_{q^{2}}$ & $(\mathbb{Z}_{p}\times \mathbb{Z}_{p})\rtimes (\mathbb{Z}_{p^{2}-1}\rtimes \mathbb{Z}_{2}) $ \\
		\hline
		33& $q|p-1$ & $(\mathbb{Z}_{p}\times \mathbb{Z}_{p}\times \mathbb{Z}_{q})\rtimes \mathbb{Z}_{q}$ & $(\mathbb{Z}_{p}\times \mathbb{Z}_{p})\rtimes (GL(2,p)\times (\mathbb{Z}_{q}\rtimes \mathbb{Z}_{q-1}))$\\
		\hline
		34& & $(\mathbb{Z}_{p}\times \mathbb{Z}_{p})\rtimes (\mathbb{Z}_{2}\times \mathbb{Z}_{2})$ & $\mathbb{Z}_{p}\rtimes ((\mathbb{Z}_{p-1}\times \mathbb{Z}_{p-1})\times \mathbb{Z}_{2})$ \\
		\hline
		35 & $q|p-1$ & $(\mathbb{Z}_{p}\times \mathbb{Z}_{p})\rtimes (\mathbb{Z}_{q}\times \mathbb{Z}_{q})$ & $(\mathbb{Z}_{p} \times \mathbb{Z}_{p})\rtimes ((\mathbb{Z}_{p-1}\times \mathbb{Z}_{p-1})\times \mathbb{Z}_{2})$ \\
		\hline
		36 & $q|p+1$, $p\not\equiv 1\Mod{q}$& $(\mathbb{Z}_{p}\times \mathbb{Z}_{p})\rtimes (\mathbb{Z}_{q}\times \mathbb{Z}_{q})$ & $(\mathbb{Z}_{p}\times \mathbb{Z}_{p})\rtimes ((\mathbb{Z}_{2}\times \mathbb{Z}_{p^{2}-1})\times (\mathbb{Z}_{q-1}\rtimes \mathbb{Z}_{q}))$ \\
		\hline
	\end{tabular}
\caption{ Structure of $Aut(G)$ of all groups $G$ of order $p^{2}q^{2}$} such that $pq\ne 6$.
\end{sidewaystable}
\vspace{.3cm}
\textbf{Acknowledgment} The first author is supported by the Senior Research Fellowship of UGC, India.

\end{document}